%% file: Main.tex
\documentclass[11pt]{elsarticle}
\usepackage{setspace}
\usepackage[plain]{fullpage}
\usepackage{algorithm}
\usepackage{algorithmicx}
\usepackage[noend]{algpseudocode}
\usepackage{algpascal}
\usepackage{epsfig}
\usepackage{enumerate}
\usepackage{paralist}
\usepackage{color}
\usepackage{srcltx}
\usepackage{amsmath, amssymb, amsthm}
\usepackage{tikz}
\usetikzlibrary{decorations.pathmorphing}
\usetikzlibrary{decorations.pathreplacing}
\usepackage{relsize}
\usepackage{graphics}
\usepackage{booktabs}
\input{gendefs}

\input{artdefs}

\input{Julien}

\newtheorem{remark}{Remark}

\begin{document}
\begin{frontmatter}
\title{Totally Disjoint Diametral Paths\\ \today}
\author[uo]{Arthur M. Farley}
\ead{art@cs.uoregon.edu}
\author[boun]{T\i naz Ekim}
\ead{tinaz.ekim@bogazici.edu.tr}

\address[uo]{University of Oregon, Eugene, Oregon 97403 USA}
\address[boun]{Bogazici University, Department of Industrial Engineering, 34342, Istanbul, Turkey}

\begin{abstract}
\input{00-Abstract}
\end{abstract}

\begin{keyword}
Graph diameter, disjoint paths, graph classes, extremal graphs.
\end{keyword}
\end{frontmatter}

\section{Introduction}\label{sec:intro}
\input{01-Introduction}

\section{Complexity of $k$-TDDP}\label{sec:complexity}
\input{02-Complexity}

\section{Maximum Number of Totally Disjoint Diametral Paths in Certain Graph Classes}\label{sec:exact}
\input{03-ExactResults}

\section{Extremal Graphs}\label{sec:extremal}
\input{04-ExtremalGraphs}

\section{Conclusion}

\input{05-Conclusion}

\bibliographystyle{abbrv}
\bibliography{References}

\end{document}

%% file: gendefs.tex
\usepackage{amsmath}
\usepackage{amssymb}
\usepackage{color}
\definecolor{red}{RGB}{255,0,0}
\definecolor{blue}{RGB}{0,0,255}
\definecolor{green}{RGB}{0,255,0}

\newcommand{\ignore}[1] {}

%% file: artdefs.tex
\newtheorem{theorem} {Theorem}
\newtheorem{lemma} {Lemma}

\newtheorem{corollary}  {Corollary}

%% file: Julien.tex
\usepackage{tikz}

%% file: 00-Abstract.tex
In this paper, we study totally disjoint diametral paths in simple connected graphs.  A diametral path in a graph is a shortest path that connects two vertices whose mutual distance is equal to the diameter of the graph. Totally disjoint paths are paths that have no vertices in common, including their end vertices.  We show that the problem of deciding whether a graph $G$ has $k$ totally disjoint diametral paths is NP-complete.  We consider restricted classes of graphs for which the problem of determining the maximum size of a set of totally disjoint diametral paths is readily solved.  We then give a linear-time algorithm for a subclass of maximal outerplanar graphs called 2-paths, define a polynomial-time algorithm for threshold graphs, and establish a structural bound for proper interval graphs. Finally, we define classes of extremal graphs with $k$ totally disjoint diametral paths of length $d$ having the fewest possible number of edges.

%% file: 01-Introduction.tex
A graph's diameter is an upper limit on the distance between vertices in a graph. Totally disjoint paths are non-interacting ways to traverse parts of a graph.  This paper puts these two notions together, exploring problems involving totally disjoint diametral paths in graphs. Such paths model scenarios where one needs multiple long-haul routes across a network that share no infrastructure whatsoever, not even endpoints, so that each route is a fully independent worst-case traversal. This is natural in fault-tolerant routing and redundant backbone design, where the diameter captures the critical (longest) communication delay and total disjointedness guarantees that no single vertex failure can disrupt more than one such critical route.

A simple graph $G=(V,E)$ is a set of vertices $V$ and a set of edges $E$, such that each edge $(u, v)$ in $E$ connects a distinct pair of vertices (i.e., \emph{neighbors}) $u$ and $v$ in $V$.  The \emph{degree} of a vertex is the number of neighbors the vertex has.  A \emph{path} of length $k$ in a simple graph G is a sequence of distinct vertices $(v_0, v_2, \ldots, v_k)$, where $(v_i, v_{i+1})$ is an edge in $E$ for $0 \leq i < k$; the path \emph{connects} $v_0$ and $v_k$.  A path of length $k$ contains $k$ edges and $k+1$ vertices. A graph $G$ is \emph{connected}, if between every pair of vertices in $V$ there exists a path connecting them.  In this paper, every graph $G$ is a simple, connected graph.

The \emph{distance} between a pair of vertices in $G$ is the length of a shortest path connecting the two vertices.  The \emph{diameter} of a graph $G$ is the greatest distance between any pair of vertices in $G$.  Pairs of vertices having distance equal to the diameter are called \emph{diametral pairs}; a shortest path between a diametral pair is called a \emph{diametral path}.  There can be more than one diametral pair in a graph and, between each such pair, there can be more than one diametral path.

In this paper, a pair of paths in a graph is \emph{totally disjoint} if the two paths have no vertex in common, \emph{including their end vertices}.  A set of paths of a connected graph is a \emph{set of totally disjoint paths} if every pair of paths in the set is totally disjoint. We denote by $k$-TDDP the problem of deciding if a given graph $G$ has a set of totally disjoint diametral paths of size $k$. The maximum size of a set of totally disjoint diametral paths in $G$ is denoted by $\#TDDP(G)$.

Interest in disjoint paths in graphs has a long history \cite{SEYMOUR1980293,even1975}, with applications to routing and commodity flows. This research primarily has focused primarily on edge-disjoint paths or vertex-disjoint paths that share the same end vertices. This line of research is mainly related to the well-known Menger's Theorem and its variants \cite{menger}. Here, unlike those studies, we focus on sets of paths having no vertices in common, including their end vertices. Requiring distinct end vertices changes the problem fundamentally: rather than measuring the connectivity between a fixed pair (as in Menger-type results), we ask how many fully independent diametral routes a network can simultaneously support. This quantity can be viewed as a measure of a network's redundant "long-distance capacity" at its critical scale. A much-studied $k$ disjoint paths problem involves finding pairwise vertex disjoint paths between $k$ specified pairs of vertices within a given graph or deciding that they do not exist (see \cite{Frank90} and \cite{Kawa2011} for a survey on the topic). There are versions of the $k$ disjoint paths problem where the paths between the given pairs of vertices are required to be of shortest length (for instance in \cite{DSP2} and \cite{DSP1}) or not (for instance, \cite{ROBERTSON199565}). These problems are at the heart of several recent studies (\cite{DSPlinear, 2DSP}) due to their connections with both combinatorial optimization problems and the theory of graph minors \cite{ROBERTSON199565}. A well-studied version is the shortest $k$ disjoint paths problem where one looks for $k$ vertex disjoint shortest paths between $k$ given pairs of vertices such that their \textit{total} distance is minimized. The decision version of this problem is NP-hard even under several restrictions; see \cite{FPT} and the references therein for a recent summary of the extensive literature on the related results. As noted, these problems have applications in multi-commodity flows and VLSI design.

In this paper, we first address the complexity of $k$-TDDP in arbitrary graphs in Section \ref{sec:complexity}. We prove that deciding if an arbitrary graph contains a set of at least $k$ totally disjoint diametral paths is NP-complete. This places $k$-TDDP among the disjoint-path problems that remain intractable in general, motivating the search for tractable structured classes that follows. In Section \ref{sec:exact}, we consider several graph classes for which one can readily determine the maximum size of a set of totally disjoint diametral paths, including efficient algorithms for 2-paths (a subclass of maximal outerplanar graphs) and threshold graphs, and establish a bound for proper interval graphs. 
Finally, we consider the related extremal graph problem in Section \ref{sec:extremal}. Given parameters $d$ and $k$, we describe classes of graphs with diameter $d$ having $k$ totally disjoint diametral paths with the minimum possible number of edges.

%% file: 02-Complexity.tex
Let $k$-TDDP be the problem of deciding if a given graph $G$ admits at least $k$ totally disjoint diametral paths. This problem is related to the \emph{k disjoint shortest paths problem}, denoted by $k$-DSP: Given a graph $G$ and $k$ pairs of distinct vertices $(s_0, t_0), (s_1, t_1), \ldots, (s_{k-1}, t_{k-1})$, determine if there exist $k$ pairwise vertex disjoint \emph{shortest} paths $p_i$ between $s_i$ and $t_i$ for each $i,  0 \leq i \leq k-1$  \cite{DSP1}. This problem is known to be $NP$-complete when $k$ is part of the input, even for planar graphs \cite{DSP1}. We will use this result to show that $k$-TDDP is NP-complete.

Developing efficient algorithms for $k$-DSP for fixed $k$ (in particular for $k=2$) is an active research area \cite{DSPlinear, 2DSP}. In a recent result, considering shortest paths where $k$ is fixed, Lochet \cite{LOCHET2021} presents an algorithm with a running time $n^{O(k^{5^k})}$ that decides whether there exist $k$ disjoint shortest paths between $k$ given pairs of vertices; this was subsequently improved to $n^{O(k!\,k)}$ by Bentert et al.\ \cite{BENTERT2021}. This result implies a polynomial time algorithm for $k$-TDDP when $k$ is fixed by considering all sets of diametral pairs of size $k$, of which there are only polynomially many. Our $NP$-hardness result for $k$-TDDP holds when $k$ is part of the input, while for every fixed $k$ the enumeration above yields a polynomial time algorithm. We note that Lochet~\cite{LOCHET2021} (see also Bentert et al.~\cite{BENTERT2021}) also showed that $k$-DSP is W[1]-hard parameterized by $k$, so for $k$-DSP the dependence of the exponent on $k$ presumably cannot be removed. Whether $k$-TDDP is likewise W[1]-hard parameterized by $k$ is open.

A related problem is the vertex disjoint paths problem, where mutually vertex disjoint, not necessarily shortest, paths between $k$ given pairs of distinct vertices must be found. This problem is known to be NP-complete \cite{karp}, and remains NP-complete on split graphs \cite{Heggernes2015finding}. In \cite{ROBERTSON199565}, Robertson and Seymour state that if $k$ is fixed, then given a graph $G$ and $k$ pairs of vertices of $G$, one can decide if there are $k$ mutually vertex disjoint paths in $G$ joining the given pairs in time $O(n^3)$ where $n$ is the number of vertices in $G$. In \cite{Heggernes2015finding}, Heggernes et al. show that this problem admits a kernel with $O(k^2)$ vertices when restricted to split graphs.

We establish the NP-completeness of $k$-TDDP in general (connected) graphs using a reduction from $k$-DSP. The reduction transforms an instance $G$ of $k$-DSP by adding new vertices attached to the vertex pairs $(s_i,t_i)$ for each $i$ with paths whose lengths are carefully chosen with respect to the diameter of $G$. Then we show that there are $k$ disjoint shortest paths in $G$ between the given pairs $(s_i,t_i)$ for each $i$ if and only if there are $k$ totally disjoint diametral paths in the newly constructed graph. We note that even though the input graph is planar, our transformation does not preserve planarity; thus we obtain the NP-completeness of $k$-TDDP for general graphs, but not for planar graphs using the result in \cite{DSP1}.

\begin{theorem}\label{thm:complexity}
$k$-TDDP is $NP$-complete for an arbitrary connected graph $G$ and positive integer $k\geq 2$. 
\end{theorem}

\begin{proof}
Given a graph $G$ and a set of $k$ paths, one can check if these $k$ paths are pairwise totally disjoint diametral paths in polynomial time. Indeed, this can be done by using well-known polynomial-time algorithms to compute shortest paths, such as executing Dijkstra's algorithm \cite{dijkstra1959note},
from all nodes or employing the Floyd-Warshall algorithm \cite{floyd1962algorithm} for all-pairs computation. Thus, $k$-TDDP is in $NP$. 
    
Given an instance $I$ of $k$-DSP which is a graph $G$ and $k$ pairs of distinct vertices $(s_i, t_i)$ for $0 \leq i \leq k-1$, we construct a graph $G'$ from $G$ by adding vertices and edges, while leaving the graph $G$ itself unchanged. Let $D$ be the diameter of the graph $G$.  For each pair $s_i,t_i$ create the following ``row" $i$. To each vertex $s_i$ connect a vertex $a_i$ by a path of length $2D+D-d_i$, where $d_i$ is the distance between $s_i$ and $t_i$ in $G$. To each vertex $t_i$ connect a vertex $b_i$ by a path of length $2D$. The distance between each $a_i$ and $b_i$ through the row’s path and $G$ is $5D$. 
    
Between each pair of rows $i,j$ connect the vertex at distance $D$ from $a_i$ to the vertex at distance $D$ from $a_j$ by a path of length $2D$. Similarly, connect the vertex at distance $D$ from $b_i$ to the vertex at distance $D$ from $b_j$ by a path of length $2D$. We call these paths between rows \textit{bridging paths} and refer to these vertices at distance $D$ from $a_i$ and from $b_i$ as \textit{bridge} vertices (shown by black circles in Figure \ref{fig:construction}). Additionally, connect the bridge vertex near $a_i$ to the bridge vertex near $b_j$ by a path of length $2D$ and the bridge vertex near $a_j$ to the bridge vertex near $b_i$ by a path of length $2D$; we refer to these paths as \textit{crossing paths}. See Figure \ref{fig:construction}. The path in row $i$ between $a_i$ and $b_i$ of length $5D$ is the shortest path between those two vertices.  Note that there is a path between $a_i$ and $b_i$ using a crossing path between row $i$ and $j$ and a bridging path back to row $i$ that has length $6D$, which is greater.

Now, we show that the row paths between every $a_i$ and $b_i$ are the only diametral paths of the graph $G’$, i.e., that the distance between every other pair of vertices is less than $5D$. By inspection, the distance between any $a_i$ and $a_j$, for $i \neq j$, is $4D$, using a bridging path, and similarly for any $b_i$ and $b_j$.  The distance between any pair $a_i$ and $b_j$, for $i \neq j$, is $4D$, using a crossing path. The distance between any other pair of vertices on these paths is less than $4D$.  The maximum distance between two vertices that both lie on the crossing and bridging paths is $4D$. The distance from $a_i$ or $b_i$ to the bridging vertex on the other side of $G$ in row $i$ is $4D$.  The distance from $a_i$ or $b_i$ to a bridging vertex in a different row $j$ is $3D$.  Thus, any vertex on a bridging or crossing path can be reached within distance less than $5D$ from any vertex in $G'$.
    
It remains to discuss the set of vertices $H$ lying between bridging vertices and graph $G$. As already noted, the distance from any $a_i$ or $b_i$ to either bridging vertex in another row $j$ is $3D$, so any vertex in $H$ can be reached in distance at most $5D-d_j$, with $d_j > 0$.  The distance between any two vertices in $H$ going through $G$ is less than $5D-d_i - d_j$; there are shorter paths between pairs, using bridging or crossing edges.

Thus, the row paths between $a_i$ and $b_i$, for $0 \leq i \leq k-1$, are the only diametral paths of $G'$. A pair of diametral paths, one in row $i$ and one in row $j$ of $G'$, are totally disjoint if and only if the shortest paths between the pair $(s_i,t_i)$ and the pair $(s_j,t_j)$ are vertex disjoint in $G$. Therefore, there are $k$ disjoint shortest paths in $G$ between the given pairs of $s$ and $t$ vertices if and only if there are $k$ totally disjoint diametral paths in $G’$.
   
It follows that $k$-TDDP is $NP$-complete.
\end{proof}

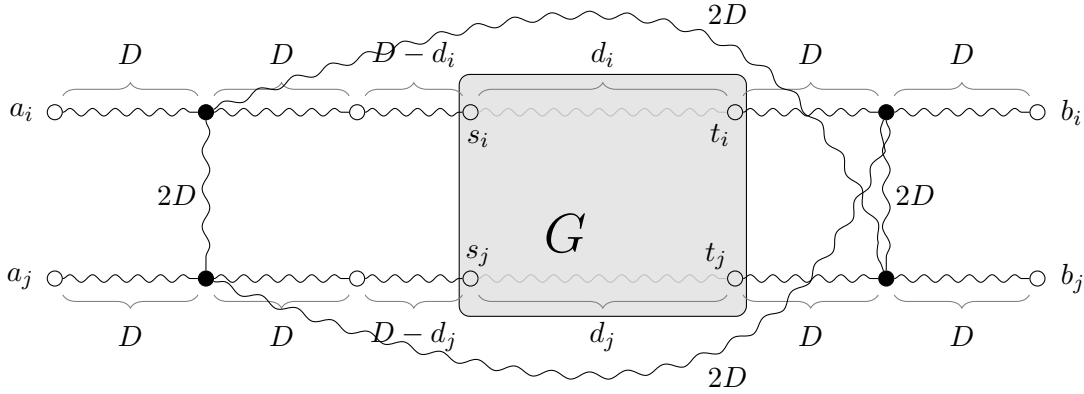
\begin{figure}[htb]
    \centering
    \begin{tikzpicture}

  % Define the layers
    \pgfdeclarelayer{background}
    \pgfsetlayers{background,main}
    
    % Nodes
    \node[draw, circle, inner sep=2pt] (ai) at (0,1.6) {};
    \node[left=1pt] at (ai.west) {$a_{i}$};

    \node[draw, circle, fill=black, inner sep=2pt] (a1) at (2,1.6) {};
    
    \node[draw, circle, inner sep=2pt] (a2) at (4,1.6) {};
    
    \node[draw, circle, inner sep=2pt] (si) at (5.5,1.6) {};
    \node[] at (5.6,1.25) {$s_{i}$};

    \node[draw, circle, inner sep=2pt] (ti) at (9,1.6) {};
    \node[] at (8.8,1.3) {$t_{i}$};
    
    \node[draw, circle, fill=black, inner sep=2pt] (b1) at (11,1.6) {};
    
    \node[draw, circle, inner sep=2pt] (bi) at (13,1.6) {};
    \node[right=2pt] at (bi.east) {$b_{i}$};
   
    \node[draw, circle, inner sep=2pt] (aj) at (0,-0.6) {};
    \node[left=1pt] at (aj.west) {$a_{j}$};

    \node[draw, circle, fill=black, inner sep=2pt] (a11) at (2,-0.6) {};
    
    \node[draw, circle, inner sep=2pt] (a22) at (4,-0.6) {};
    
    \node[draw, circle, inner sep=2pt] (sj) at (5.5,-0.6) {};
    \node[] at (5.6,-0.3) {$s_j$};
    
    \node[draw, circle, inner sep=2pt] (tj) at (9,-0.6) {};
    \node[] at (8.75,-0.3) {$t_{j}$};

    \node[draw, circle, fill=black, inner sep=2pt] (b11) at (11,-0.6) {};
    
    \node[draw, circle, inner sep=2pt] (bj) at (13,-0.6) {};
    \node[right=2pt] at (bj.east) {$b_{j}$};
    
    % Graph G
   
    \begin{pgfonlayer}{background}
        \node[draw, rounded corners, fill=gray!20, minimum width=3.8cm, minimum height=3.2cm] at (7.25,0.5) {};
    \end{pgfonlayer}
    
    \node at (6.75,0) {\larger[4]{$G$}};

    % Horizontal Edges
    \draw[decorate, decoration={snake, amplitude=0.5mm, segment length=3mm}] (ai) -- (a1);
    \draw [decorate, gray, decoration={brace,amplitude=5pt, raise=1.3ex}]
  (ai) -- (a1) node[midway,black, yshift=2em]{$D$};
    \draw [decorate, gray, decoration={brace,amplitude=5pt,raise=1.3ex}]
  (a1) -- (a2) node[midway, black, yshift=2em]{$D$};
  \draw [decorate, gray, decoration={brace,amplitude=5pt,raise=1.3ex}]
  (a2) -- (si) node[midway, black, yshift=2em]{$D-d_i$};
    \draw [decorate, decoration={snake, amplitude=0.5mm, segment length=3mm}] (a1) -- (a2);
    \draw[decorate, decoration={snake, amplitude=0.5mm, segment length=3mm}] (a2) -- (si);
     \draw[decorate, decoration={snake, amplitude=0.5mm, segment length=3mm}] (ti) -- (b1);
     \draw[decorate, decoration={snake, amplitude=0.5mm, segment length=3mm}] (b1) -- (bi);
    \draw [decorate, gray, decoration={brace,  amplitude=5pt,mirror, raise=1.3ex}]
  (bi) -- (b1) node[midway,black, yshift=2em]{$D$};
    \draw [decorate, gray, decoration={brace,amplitude=5pt, mirror, raise=1.3ex}]
  (b1) -- (ti) node[black, midway,yshift=2em]{$ D $};
  \draw[decorate, lightgray, decoration={snake, amplitude=0.5mm, segment length=3mm}] (si) -- (ti);
    \draw [decorate, gray, decoration={brace,  amplitude=5pt, raise=1.3ex}]
  (si) -- (ti) node[midway,black, yshift=2em]{$d_i$};

  %LINE j

    \draw[decorate, decoration={snake, amplitude=0.5mm, segment length=3mm}] (aj) --  (a11);
    \draw [decorate, decoration={snake, amplitude=0.5mm, segment length=3mm}] (a11) -- (a22);
    \draw[decorate, decoration={snake, amplitude=0.5mm, segment length=3mm}] (a22) -- (sj);
    \draw[decorate, lightgray, decoration={snake, amplitude=0.5mm, segment length=3mm}] (sj) -- (tj);
     \draw[decorate, decoration={snake, amplitude=0.5mm, segment length=3mm}] (tj) -- (b11);
    \draw[decorate, decoration={snake, amplitude=0.5mm, segment length=3mm}] (b11) --  (bj);

    %line j lenghts
    \draw [decorate, gray, decoration={brace,amplitude=5pt, mirror, raise=1.3ex}]
  (aj) -- (a11) node[midway,black, yshift=-2em]{$D$};
    \draw [decorate, gray, decoration={brace,amplitude=5pt, mirror, raise=1.3ex}]
  (a11) -- (a22) node[midway, black, yshift=-2em]{$D$};
  \draw [decorate, gray, decoration={brace,amplitude=5pt,mirror,raise=1.3ex}]
  (a22) -- (sj) node[midway, black, yshift=-2em]{$D-d_j$};

  \draw [decorate, gray, decoration={brace,  amplitude=5pt, raise=1.3ex}]
  (bj) -- (b11) node[midway,black, yshift=-2em]{$D$};
    \draw [decorate, gray, decoration={brace,amplitude=5pt, raise=1.3ex}]
  (b11) -- (tj) node[black, midway,yshift=-2em]{$ D $};
      \draw [decorate, gray, decoration={brace,  amplitude=5pt, mirror, raise=1.3ex}]
  (sj) -- (tj) node[midway,black, yshift=-2em]{$d_j$};
    
    % Vertical and Diagonal Edges

    \draw[decorate, decoration={snake, amplitude=0.5mm, segment length=5mm}] (a1) .. controls (7, 3.5) and (9.5, 3.5) ..  (b11);

    \node at (8.9,2.9) {$2D$};
     \node at (8.9,-1.9) {$2D$};
    \draw[decorate, decoration={snake, amplitude=0.5mm, segment length=5mm}]  (a11) .. controls (7, -2.5) and (9.5, -2.5) .. (b1);
    \draw[decorate, decoration={snake, amplitude=0.5mm, segment length=5mm}] (a1) --  (a11)  node[midway,black, xshift=-1em]{$2D$};
    \draw[decorate, decoration={snake, amplitude=0.5mm, segment length=5mm}] (b1) --  (b11) node[midway,black, xshift=1em]{$2D$};

   % \draw [decorate, gray, decoration={brace, mirror, amplitude=5pt, raise=1.3ex}]
  (%ai) -- (aj) node[midway,black, xshift=-2em]{$N$};
  %\draw [decorate, gray, decoration={brace,  amplitude=5pt, raise=1.3ex}]
  %(bi) -- (bj) node[midway,black, xshift=2em]{$N$};

\end{tikzpicture}
    \caption{Two diametral paths between pairs $(a_i,b_i)$ and $(a_j, b_j)$ in $G'$ where curly lines show paths with length denoted next to them. Gray area depicts the graph $G$. Likewise, the edges on the gray paths between pairs $(s_i,t_i)$ and $(s_j,t_j)$ belong to $G$ whereas all the others show paths consisting of the edges in $G'\setminus G$.}
    \label{fig:construction}
\end{figure}

%% file: 03-ExactResults.tex
In this section, we investigate several classes of graphs for which we can determine or bound the maximum size of a set of totally disjoint diametral paths, $\#TDDP(G)$. As a first example, we consider tree graphs.  A $tree$ $T_n$ is a connected acyclic graph on $n$ vertices having $n-1$ edges. 

\begin{theorem} \label{thm:tree}
 $\#TDDP(T_n) = 1$, for all $n\geq1$.
\end{theorem}

\begin{proof}
    The \emph{vertex center of a graph} is the set of vertices that minimize the maximum distance to other vertices of the graph. It is known that the vertex center of a tree $T_n$ is a single vertex or two neighboring vertices and that every diametral path of a tree goes through its vertex center. Therefore, $\#TDDP(T_n) = 1$, for all $n\geq1$.
\end{proof}

\subsection{Parameterized Graphs} 

Here we consider classes of graphs that have a fixed size and structure given values of one or more integer parameters.  Cycle $C_n$ can be described as a set of distinct vertices $v_0, v_1, v_2, \ldots, v_{n-1}$ with edges $(v_i, v_{i+1})$ for $0 \leq i < n-1$, plus edge $(v_{n-1}, v_0)$.  We define $EC_n$, for odd $n$, to be $C_n$ enhanced by two edges and a new vertex $u$ with edges $(v_{\lfloor n/2 \rfloor}, u)$ and $(u, v_{\lceil n/2 \rceil +1})$. See Figure \ref{fig:CC22} for $EC_7$. A $CC_{2n}$, for $n>1$, has $2n$ vertices $v_0, v_1, \ldots, v_{2n-1}$ with edges $(v_i, v_{i+1})$ for $0 \leq i<2n-1$, edge $(v_{2n-1}, v_0)$, and chordal edges $(v_i, v_{i+n})$ for $i<n$.  All vertices look the same in a $CC_{2n}$, as it is a symmetric graph. See Figure \ref{fig:CC22} for $CC_{22}$. The Rectangular Grid Graph $RGG_{mn}$ is the graph whose vertices correspond to the points in the plane with integer coordinates, where $x$-coordinates are in the range $1, ..., n$, and $y$-coordinates are in the range $1, ..., m$; we designate the vertices as $v_{x,y}$. Two vertices are connected by an edge whenever the corresponding points are at distance $1$ in the plane.  Finally, the complete bipartite graph $B_{m,n}$ consists of two sets of vertices of sizes $m$ and $n$, where every vertex of each set has an edge connecting it to every vertex in the other set, with no edges between vertices in the same set. 

\begin{theorem}\label{thm:fixed}
    \begin{enumerate}[i.]
    
\item Let $K_n$ be the complete graph on $n$ vertices, then $\#TDDP(K_n) = \lfloor n/2 \rfloor$, for all $n>1$.

\item Let $C_n$ be the cycle on $n$ vertices, then $\#TDDP(C_n) = 1$, for all $n>2$.

\item Let $EC_n$ be an enhanced cycle for odd $n$, then $\#TDDP(EC_n) = 2$, for all odd $n>2$.

\item Let $CC_{2n}$ be an even cycle with diameter chords, then $\#TDDP(CC_{2n}) = 2$, for $2<n<6$ and $7$;
		  	    $\#TDDP(CC_{2n}) = 3$, for $n\geq6$, except $7$.
\item Let $RGG_{m,n}$ be a Rectangular Grid Graph, then $\#TDDP(RGG_{m,n}) = 1$.

\item Let $B_{m,n}$ be a complete bipartite graph where $m\geq n$, then 
				$\#TDDP(B_{m,n}) = n$, for $m\geq 2n$   
    
				$\#TDDP(B_{m,n}) =  \lfloor (m+n)/3 \rfloor$, for $m<2n$.

    \end{enumerate}
\end{theorem}

\begin{proof}

    \begin{enumerate}[i.]
   
\item Since every vertex in $K_n$ is connected directly to every other vertex, the diameter of $K_n$ is $1$.  Every diametral path includes 2 vertices. Therefore, $\#TDDP(K_n) = \lfloor n/2 \rfloor$, for all $n>1$.

\item 	The diameter of cycle $C_n$ is $\lfloor n/2 \rfloor$; each diametral path contains $\lfloor n/2 \rfloor + 1$ vertices, which is more than half of the vertices in the cycle. Therefore, $\#TDDP(C_n) = 1$, for all $n>2$.

\item 	The diameter of $EC_n$ is $\lfloor n/2 \rfloor$.  Note that vertex $u$ is at diameter distance from both $v_0$ and $v_1$. There are two disjoint diametral paths $(v_0, v_{n-1}) (v_{n-1}, v_{n-2}), \ldots, (v_{\lceil n/2 \rceil+1},v_{\lceil n/2 \rceil})$ and the path $(v_1, v_2) (v_2,v_3), \ldots, (v_{\lfloor n/2 \rfloor}, u)$. Therefore, $\#TDDP(EC_n) = 2$, for all odd $n>2$.
    
\item   $CC_{2m}$ has $3m$ edges.  The diameter of $CC_{2m}$ is $ \lceil m/2 \rceil $; the number of vertices in a diametral path is thus $ \lceil m/2 \rceil + 1$, which is more than one-quarter of the vertices in the graph.    

 	Once $m$ is large enough, three totally disjoint diametral paths can be realized in a $CC_{2m}$ by taking consecutive diametral paths around the cycle. As an example, see Figure \ref{fig:CC22} for $m=11$ (22 vertices) where the diameter is $\lceil m/2 \rceil = 6$. The number of vertices included in a diametral path is $7$; thus, $\#TDDP(CC_{22}) = 3$. Given the number of vertices in a diametral path, there can never be more than three disjoint paths. 
  
	Therefore, 	$\#TDDP(CC_{2m}) = 2$, for $2<m<6$ and $7$;
		  	    $\#TDDP(CC_{2m}) = 3$, for $m\geq6$, except $7$.

\item   $RGG_{m,n}$ has diameter $(m-1)+(n-1)$ with two diametral pairs, $(v_{1,1}, v_{n,m})$ and $(v_{1,m} , v_{n,1})$ In an $RGG_{m,n}$ every path between $v_{1,1}$ and $v_{n,m}$ shares a vertex with every path between $v_{1,m}$ and $v_{n,1}$.  Each “diagonal” path separates the other corners from each other. Therefore, $\#TDDP(RGG_{m,n}) = 1$.

\item The diameter of $B_{m,n}$ is $2$.  Each diametral path goes from a vertex in one set to a vertex in the other set and back, involving two vertices in one set and one in the other.  When $m$ is greater then $2n$, disjoint diametral paths can always use two vertices from the larger set and one from the smaller.  Otherwise, one uses two from either set and one from the other, as needed.

	Therefore, in a $B_{m,n}$ where $m\geq n$,
 
				$\#TDDP(B_{m,n}) = n$, when $m\geq 2n$   
    
				$\#TDDP(B_{m,n}) =  \lfloor (m+n)/3 \rfloor$, when $m<2n$.
\end{enumerate}
\end{proof}

\begin{figure}[htb]
    \centering
\begin{tikzpicture}
 % First figure on the left
    \begin{scope}[shift={(-4,0)}] % Shift this scope to the left

    % Define the nodes in a circular arrangement with a smaller radius
    \foreach \i in {0,...,21} {
        \node[draw, circle, inner sep=2pt] (v\i) at ({360/22 * \i}:2cm) {};
        \node at ({360/22 * \i}:2.5cm) {$v_{\i}$};
    }

    % Draw the edges for the cycle
    \foreach \i in {0,...,20} {
        \pgfmathtruncatemacro{\next}{\i + 1}
        \draw (v\i) -- (v\next);
    }
    \draw (v21) -- (v0);

    % Draw the additional edges
    \foreach \i in {0,...,10} {
        \pgfmathtruncatemacro{\j}{\i + 11}
        \draw (v\i) -- (v\j);
    }

    % Caption and label for the first figure
        %\node[below=3cm, align=center] {Figure 1: First figure\\ \label{fig:figure1}};
\end{scope}

 % Second figure on the right
    \begin{scope}[shift={(4,0)}] % Shift this scope to the right
        % Your second TikZ figure code here
         % Define the nodes in a circular arrangement with a radius of 2cm
    \foreach \i in {0,...,6} {
        \pgfmathtruncatemacro{\j}{mod(\i+1,7)}
        \node[draw, circle, inner sep=2pt] (v\j) at ({360/7 * \i}:1.5cm) {};
        \node at ({360/7 * \i}:2cm) {$v_{\j}$};
    }

    % Draw the edges for the cycle
    \foreach \i in {0,...,5} {
        \pgfmathtruncatemacro{\next}{\i + 1}
        \draw (v\i) -- (v\next);
    }
    \draw (v6) -- (v0);

    % Define the outside node u close to v_4 and equally distant from v_3 and v_5
    \node[draw, circle, inner sep=2pt] (u) at (-2.4, 1) {};
    \node[above=2pt] at (u.north) {$u$};

    % Draw the edges from u to v3 and v5
    \draw (u) -- (v3);
    \draw (u) -- (v5);

     % Caption and label for the second figure
        %\node[below=3cm, align=center] {Figure 2: Second figure\\ \label{fig:figure2}};
    \end{scope}

\end{tikzpicture}
\caption{The graph $CC_{22}$ on the left and $EC_7$ on the right.}
        \label{fig:CC22}
\end{figure}
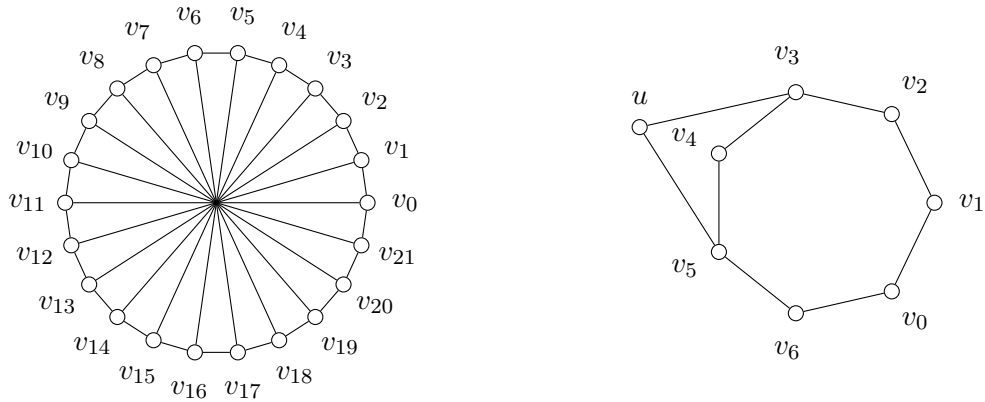
 
Observing that $\#TDDP$ is at most the number of vertices divided by one plus the diameter, the following upper bound on $\#TDDP$ applies to all graphs. 

\begin{remark}\label{rem:general}
For any graph $G$ with $n$ vertices and diameter $D$, we have $\#TDDP(G)\leq n/(D+1)$.
\end{remark}

The hypercube of dimension $n$, denoted by $Q_n$, has its vertex set representing distinct binary strings of length $n$. Each vertex shares an edge with $n$ other vertices, which are those that differ in one bit value. 

\begin{corollary}\label{cor:hypercube}
Let $Q_n$ be a hypercube of dimension $n$, then we have $\#TDDP(Q_n) \leq 2^n/(n+1)$.
\end{corollary}
\begin{proof}
The number of vertices in $Q_n$ is $2^n$. Each vertex has a unique diametral partner, which is the vertex having all differing bit values. As such, the diameter of $Q_n$ is $n$, and each diametral path uses $n+1$ vertices. By Remark \ref{rem:general}, we have $\#TDDP(Q_n) \leq 2^n/(n+1)$.
\end{proof}
We have the following exact values of $\#TDDP(Q_n)$ for small $n\leq 6$ where the optimality follows from the upper bound in Corollary \ref{cor:hypercube}. We observe that this upper bound is tight for $n\leq 6$ except $n=4$. Finding an efficient way to determine $\#TDDP(Q_n)$ for larger $n$ or proving that the bound given above is tight for all $n$ other than $4$ are open problems.

\begin{itemize}
 
		\item $\#TDDP(Q_2)=1$, as $(00, 01, 11)$ uses 3 vertices. 
  
		\item $\#TDDP(Q_3)=2$, as (101, 001, 011, 010) and (000, 100, 110, 111) use all eight vertices in the two diametral paths.
  
		\item $\#TDDP(Q_4) =2$ with (0000, 1000, 1100, 1110, 1111) and (0100, 0101, 0001, 1001, 1011) consists of two totally disjoint diametral paths. By an exhaustive search, we find that there are no sets of three totally disjoint diametral paths in $Q_4$.

        \item $\#TDDP(Q_5)=5$ with the following five totally disjoint diametral paths: (00000, 00001, 00011, 00111, 01111, 11111); (00010, 00110, 00100, 00101, 01101, 11101); (01000, 01001, 01011, 11011, 10011, 10111); (01010, 11010, 10010, 10000, 10100, 10101);(01110, 01100, 11100, 11000, 11001, 10001)

        \item $\#TDDP(Q_6)=9$ with the following nine totally disjoint diametral paths:\\
        (000100, 000000, 000010, 010010, 110010, 111010, 111011);\\
        (010000, 010100, 110100, 110101, 110111, 111111, 101111);\\
        (010101, 010001, 000001, 000011, 001011, 101011, 101010);\\
        (011001, 011000, 011100, 011110, 010110, 000110, 100110);\\
        (011011, 011010, 001010, 001110, 101110, 101100, 100100);\\
        (100010, 100011, 110011, 010011, 010111, 011111, 011101);\\
        (110000, 100000, 101000, 001000, 001100, 001101, 001111);\\
        (110110, 111110, 111100, 111101, 101101, 101001, 001001);\\
        (111000, 111001, 110001, 100001, 100101, 100111, 000111)
    %\item \red{$15\leq \#TDDP(Q_7) \leq 16$ as we can give 15 verified totally disjoint diametral paths and the upper bound 16 follows from Corollary \ref{cor:hypercube}. WE CAN REMOVE THIS LAST ITEM}
        
\end{itemize}

\subsection{2-path Graphs}
We show that we can efficiently determine $\#TDDP$ for a subclass of maximal outerplanar graphs, called 2-paths. A graph is \textit{outerplanar} if it can be drawn on the plane so that no edges are crossing, and all its vertices lie on the exterior face.  A \textit{maximal outerplanar graph (MOP)} is an outerplanar graph such that adding a new edge would destroy outerplanarity. A maximal outerplanar graph can be seen as a fully triangulated cycle. A MOP $M$ can be constructed by starting from a triangle graph where all edges are \emph{exterior edges} (i.e., lie on the exterior face) and successively adding a new vertex connected by edges to the two end vertices of an exterior edge $e$.  The newly added edges become exterior edges and the edge $e$ becomes an \emph{interior edge}. In a MOP, each interior edge $e$ separates the MOP into two MOP subgraphs, each one including $e$ as an exterior edge.

The class of \textit{2-paths} is a subclass of MOPs having only two vertices of degree 2. A 2-path can be constructed as follows. Start from a MOP on $4$ vertices, which is uniquely defined as a 4-cycle with a diagonal chord; let the two vertices of degree 2 be called $v_l$ and $v_r$, respectively. Without loss of generality, we will consider a 2-path to be constructed from left to right, always replacing the right degree 2 vertex $v_r$ of the graph.  Add a new vertex $u$ connected to $v_r$ and one of its neighbors $w$. The new vertex $u$ replaces $v_r$ as the rightmost degree 2 vertex. 

Let $G$ be a 2-path and consider its left to right construction order.  After a fifth vertex is added to $G$, there remains one vertex of degree 3 that is adjacent to $v_l$ that we denote by $v’_l$.  These two vertices will not be involved in later additions of vertices; $v_l$ and $v’_l$ are called the \textit{left vertices} of the graph.  Similarly, we define the \textit{right vertices} $v_r$ and $v’_r$, where $v’_r$ is the unique neighbor of $v_r$ having degree $3$ (see Figure \ref{fig:2path}). 

\begin{lemma} \label{lem:pair}
Let $G$ be a 2-path with $5$ or more vertices. The pair $[v_l,v_r]$ is a diametral pair of $G$.
\end{lemma}

\begin{proof}
We show by induction that the pair $[v_l,v_r]$ is a diametral pair of $G$. It is true in the 2-path with 5 vertices. When a new vertex $u$ is added adjacent to $v_r$ and a vertex $w$, either there exists a diametral path from $v_r$ to $v_l$  through $w$, implying that $u$ is at the same distance from $v_l$ as $v_r$; or, if a diametral path does not pass through $w$, $u$ is at a greater distance than $v_r$. In either case, $[u, v_l]$ is a diametral pair with new vertex $v$ being the new $v_r$.
\end{proof}

A 2-path with diameter 2 is called a \textit{fan}. The graph on the left in Figure \ref{fig:2path} shows a fan with 6 vertices and having two totally disjoint diametral paths.

\begin{theorem}\label{thm:D2}
Given a 2-path $G$ with diameter 2 on $n$ vertices, $\#TDDP(G)=\lfloor(n/3)\rfloor$.  
\end{theorem} 
\begin{proof}
Since the diameter $D=2$, the two degree 2 vertices $v_l$ and $v_r$ have a common neighbor, say $u$, where $(v_l,u)$ and $(v_r,u)$ are external edges. Vertex $u$ is a universal vertex.
One diametral path of $G$ is from $v_l$ to $v_r$ through vertex $u$.  The other totally disjoint diametral paths can be realized as three consecutive, previously unselected vertices connected by two exterior edges of $G$.  
\end{proof}

As we construct a 2-path that has diameter greater than $2$, there comes a point when a new vertex $v$ is connected to an edge $e_b$ that has both end vertices at distance $2$ from $v_l$; the edge $e_b$ becomes an interior edge as vertex $v$, now $v_r$ at distance $3$ from $v_l$, is added to the 2-path. We call edge $e_b$ the \textit{barrier edge} of $G$ as it separates left and right parts of the graph. See the graph on the right in Figure \ref{fig:2path}, which depicts a 2-path $G$ with diameter $4$ where $\#TDDP(G)=2$, and $e_b$ is shown by a bold edge.

\begin{theorem}\label{thm:barrier}
 Every diametral path of a 2-path graph having diameter greater than 2 includes an end vertex of the barrier edge.
\end{theorem} 

\begin{proof}
The statement is true when the barrier edge is established.  As we add a new vertex $v_r$ to the right end of the graph, the distance from $v_r$ to the left vertices, $v_l$ or $v_l'$, stays the same or grows. This distance will be greater than the distance between any pair of vertices on one side of the barrier edge.
\end{proof}
 Therefore, we have the following corollary.
\begin{corollary}\label{cor:2path}
Given a 2-path graph $G$ with diameter greater than $2$, $\#TDDP(G) \leq 2$.
\end{corollary}

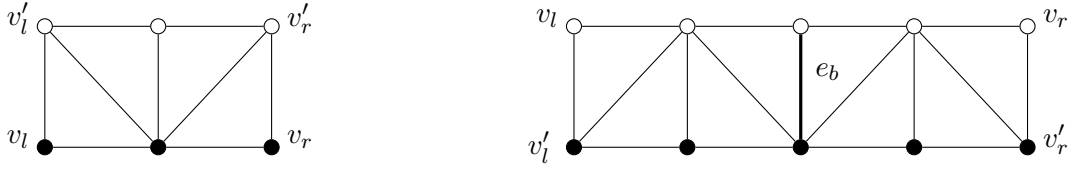
\begin{figure}[htb]
    \centering
    \begin{tikzpicture}

   \begin{scope}[shift={(-4,0)}]
    %Fan
  % Nodes
    \node[draw, circle, inner sep=2pt] (p10) at (0,1.6) {};
    \node[left=2pt] at (p10.north) {$v'_l$};

    \node[draw, circle, inner sep=2pt] (p11) at (1.5,1.6) {};

    \node[draw, circle, inner sep=2pt] (p12) at (3,1.6) {};
     \node[right=2pt] at (p12.north) {$v'_r$};

    \node[draw, circle, fill, inner sep=2pt] (p20) at (0,0) {};
     \node[left=2pt] at (p20.north) {$v_l$};
    
    \node[draw, circle, fill, inner sep=2pt] (p21) at (1.5,0) {};

    \node[draw, circle, fill, inner sep=2pt] (p22) at (3,0) {};
    \node[right=2pt] at (p22.north) {$v_r$};

   %Edges
    \draw (p10) -- (p11) -- (p12);
    \draw (p20) -- (p21) -- (p22);
    \draw (p20) -- (p10);
    \draw (p21) -- (p10);
    \draw (p21) -- (p11);
    \draw (p21) -- (p12);
    \draw (p22) -- (p12);
    
   \end{scope}

   \begin{scope}[shift={(3,0)}]
    
    % Nodes
    \node[draw, circle, inner sep=2pt] (p10) at (0,1.6) {};
    \node[left=2pt] at (p10.north) {$v_l$};

    \node[draw, circle, inner sep=2pt] (p11) at (1.5,1.6) {};

    \node[draw, circle, inner sep=2pt] (p12) at (3,1.6) {};

    \node[draw, circle, inner sep=2pt] (p13) at (4.5,1.6) {};

    \node[draw, circle, inner sep=2pt] (p14) at (6,1.6) {};
    \node[right=2pt] at (p14.north) {$v_r$};

    \node[draw, circle, fill, inner sep=2pt] (p20) at (0,0) {};
    \node[left=2pt] at (p20.west) {$v'_l$};

    \node[draw, circle, fill, inner sep=2pt] (p21) at (1.5,0) {};

    \node[draw, circle, fill, inner sep=2pt] (p22) at (3,0) {};

    \node[draw, circle, fill, inner sep=2pt] (p23) at (4.5,0) {};

    \node[draw, circle, fill, inner sep=2pt] (p24) at (6,0) {};
    \node[right=2pt] at (p24.north) {$v'_r$};
    \node[right=1.7pt] at (3,1) {$e_b$};

    % Horizontal Edges
    \draw (p10) -- (p11) -- (p12) -- (p13) -- (p14);
    \draw (p20) -- (p21) -- (p22) -- (p23) -- (p24);

    % Vertical and Diagonal Edges
    
    \draw (p10) -- (p20);
    \draw (p11) -- (p21);
    \draw[very thick] (p12) -- (p22);
    \draw (p13) -- (p23);
    \draw (p14) -- (p24);

    \draw (p11) -- (p20);
    \draw (p11) -- (p22);
    \draw (p13) -- (p22);
    \draw (p13) -- (p24);
    \end{scope}
\end{tikzpicture}
    \caption{A fan with 6 vertices in the left and a 2-path $G$ with diameter 4 in the right. Both graphs have $\#TDDP(G)=2$ with one diametral path formed by black vertices and the other by white vertices.}
        \label{fig:2path}
\end{figure}

\begin{lemma}\label{lem:ext}
For a 2-path graph $G$ with diameter greater than 2, if $\#TDDP(G)=2$, then the two totally disjoint diametral paths include only exterior edges of $G$.
\end{lemma}

\begin{proof}
Assume that a diametral path contains an interior edge $e$.  Edge $e$ separates $G$ into two subgraphs. After removing edge $e$ and its end vertices, neither subgraph can contain a pair of vertices at a distance equal to the diameter of $G$. It follows that if $\#TDDP(G)=2$, then the two totally disjoint diametral paths include only exterior edges of $G$.
\end{proof}

\begin{figure}[htb]
    \centering
    \begin{tikzpicture}

   \begin{scope}[shift={(-4,0)}]
    
    % Nodes
    \node[draw, circle, inner sep=2pt] (p10) at (0,1.6) {};
    \node[left=2pt] at (p10.north) {$v'_l$};

    \node[draw, circle, inner sep=2pt] (p11) at (1.5,1.6) {};

    \node[draw, circle, inner sep=2pt] (p12) at (3,1.6) {};

    \node[draw, circle, inner sep=2pt] (p13) at (4.5,1.6) {};
        \node[right=2pt] at (p13.north) {$v_r$};

    \node[draw, circle, fill, inner sep=2pt] (p20) at (0,0) {};
    \node[left=2pt] at (p20.west) {$v_l$};

    \node[draw, circle, fill, inner sep=2pt] (p21) at (1.5,0) {};

    \node[draw, circle, fill, inner sep=2pt] (p22) at (3,0) {};

    \node[draw, circle, fill, inner sep=2pt] (p23) at (4.5,0) {};
    \node[right=2pt] at (p23.north) {$v'_r$};

    % Horizontal Edges
    \draw (p10) -- (p11) -- (p12) -- (p13);
    \draw (p20) -- (p21) -- (p22) -- (p23);

    % Vertical and Diagonal Edges
    
    \draw (p10) -- (p20);
    \draw (p11) -- (p21);
    \draw (p12) -- (p22);
    \draw (p12) [very thick] -- (p13);
    \draw (p20) [very thick] -- (p21);
    \draw (p13) -- (p23);

    \draw (p10) -- (p21);
    \draw (p12) [very thick]-- (p21);
    \draw (p12) -- (p23);
    
    \end{scope}

   \begin{scope}[shift={(3,0)}]
    
    % Nodes
    \node[draw, circle, inner sep=2pt] (p10) at (0,1.6) {};
    \node[left=2pt] at (p10.north) {$v_l$};

    \node[draw, circle, inner sep=2pt] (p11) at (1.5,1.6) {};

    \node[draw, circle, inner sep=2pt] (p12) at (3,1.6) {};

    \node[draw, circle, inner sep=2pt] (p14) at (4.5,1.6) {};
    \node[right=2pt] at (p14.north) {$v_r$};

    \node[draw, circle, inner sep=2pt] (p20) at (0,0) {};
    \node[left=2pt] at (p20.west) {$v'_l$};

    \node[draw, circle, inner sep=2pt] (p21) at (1.5,0) {};

    \node[draw, circle, inner sep=2pt] (p22) at (3,0) {};

    \node[draw, circle, inner sep=2pt] (p23) at (4.5,0) {};

    \node[draw, circle, inner sep=2pt] (p24) at (6,0) {};
    \node[right=2pt] at (p24.north) {$v'_r$};

    % Horizontal Edges
    \draw (p10) -- (p11) -- (p12) -- (p14);
    \draw (p20) -- (p21) -- (p22) -- (p23) -- (p24);

    % Vertical and Diagonal Edges
    \draw (p20) [very thick] -- (p21);
    \draw (p10) -- (p20);
    \draw (p12) [very thick] -- (p21);
    \draw (p12) -- (p22);
    \draw (p12) -- (p23);
    \draw (p12) [very thick] -- (p24);
    \draw (p14) -- (p24);

    \draw (p11) -- (p20);
    \draw (p11) -- (p21);
    
    \end{scope}
\end{tikzpicture}
    \caption{Various cases for diametral paths.}
        \label{fig:diametralpaths}
\end{figure}
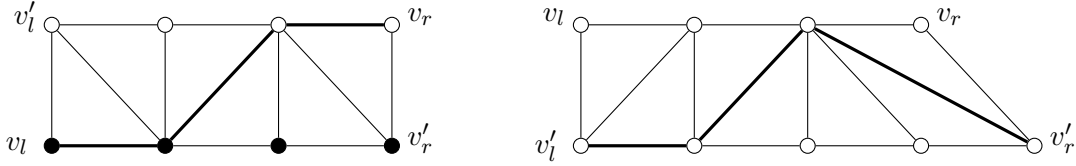

The graph on the left in Figure \ref{fig:diametralpaths} shows that while $[v_l, v_r]$ is always a diametral pair, the two totally disjoint diametral
paths (one in black, the other in white vertices) using exterior edges may be between the pairs $[v'_l, v_r]$ and $[v_l,v'_r]$. The graph on the right in Figure \ref{fig:diametralpaths} shows a 2-path having diameter $3$ with $\#TDDP(G)=1$; in order for vertex $v_r'$ to reach a left vertex, it an interior, “shortcut” edge must be used.

We can use the results of the two lemmas above to derive an efficient algorithm determining $\#TDDP(G)$ for a 2-path $G$.

\begin{theorem}
Given a 2-path graph $G$, $\#TDDP(G)$ can be determined in time linear with the number of vertices in $G$.
\end{theorem}
\begin{proof}
As the graph is constructed, we associate with each new vertex its distance to $v_l$ and $v_l'$, being one more than the minimum distance of its two neighbors to that vertex in each case.  Each vertex is incident to two exterior edges when added, the exterior edges adjacent to a vertex can be readily updated as new vertices and edges are added. When the graph is complete, the distance from $v_r$ to $v_l$ is the diameter, as shown in Lemma \ref{lem:pair}. If the diameter equals $2$, then we identify diametral paths as indicated in the proof of Theorem \ref{thm:D2}. Otherwise, if the distance from $v_r'$ to $v_l'$ or to $v_l$ is less than the diameter, then $\#TDDP(G)=1$, as $[v_l,v_r]$ is the only diametral pair. Otherwise, from vertex $v_r$, traverse exterior edges in the direction opposite from $v_r'$ to determine if there is a diametral path of exterior edges to either $v_l$  or $v_l'$.  If not, then $\#TDDP(G)=1$, as every diametral path from $v_r$ includes an interior edge. If an exterior diametral path is found, then traverse exterior edges of the graph from $v_r'$ in direction away from $v_r$ to see if there is a diametral path to the other left vertex using only exterior edges.  If one is found, then $\#TDDP(G)=2$, else $\#TDDP(G)=1$.
\end{proof}

We have described above an efficient algorithm to compute $\#TDDP$ for the subclass of MOPs called 2-paths.  The properties and complexity of determining totally disjoint diametral paths in MOPs and in the more general graph class of 2-trees (i.e., where a new vertex can be connected by two edges to the ends of any existing edge) remain open problems.

\subsection{Threshold Graphs} 
Threshold graphs have structural properties that allow us to provide upper bounds on $\#TDDP$ or to compute it exactly in special cases. A \emph{threshold graph} is a subclass of split graphs that can be constructed by repeated applications of the following two operations: $(1)$ addition of a single isolated vertex; $(2)$ addition of a single dominating vertex, i.e., a single vertex that is connected to all other vertices of the (current) graph. Let $D$ be the set of dominating vertices and $I$ be the set of isolated (or independent) vertices; noting that the terms dominating and isolated refer to their property by the time they were added to the graph. Given a graph, it can be recognized whether it is a threshold graph or not in linear time. Moreover, if it is a threshold graph, we can determine in linear time the order in which vertices are added to construct it using the two operations described above \cite{threshold}.

 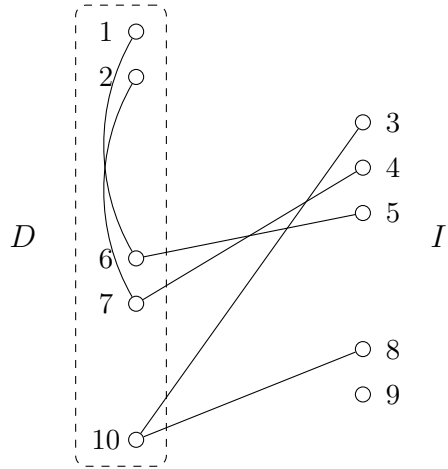
\begin{figure}[htb]
    \centering
\begin{tikzpicture}

\node at (-1.5, 1.3) {\large $D$};

\node[draw, dashed, rounded corners, minimum width=1.2cm, minimum height=6.1cm] at (-0.2,1.3) {};

% nodes

\node at (-0.4,4) {1};
\node[circle, draw, inner sep=2pt] (1) at (0, 4) {}; 

\node at (-0.4,3.4) {2};
\node[circle, draw, inner sep=2pt] (2) at (0, 3.4) {};

\node at (3.4,2.8) {3};
\node[circle, draw, inner sep=2pt] (3) at (3, 2.8) {};

\node at (3.4,2.2) {4};
\node[circle, draw, inner sep=2pt] (4) at (3, 2.2) {};

\node at (3.4,1.6) {5};
\node[circle, draw, inner sep=2pt] (5) at (3, 1.6) {};

\node at (-0.4,1) {6};
\node[circle, draw, inner sep=2pt] (6) at (0, 1) {};

\node at (-0.4,0.4) {7};
\node[circle, draw, inner sep=2pt] (7) at (0, 0.4) {};

\node at (3.4,-0.2) {8};
\node[circle, draw, inner sep=2pt] (8) at (3, -0.2) {};

\node at (3.4,-0.8) {9};
\node[circle, draw, inner sep=2pt] (9) at (3, -0.8) {};

\node at (-0.4,-1.4) {10};
\node[circle, draw, inner sep=2pt] (10) at (0,-1.4) {};

\node at (4, 1.3) {\large $I$};

% Edges

\draw (3) -- (10);
\path [bend right, distance=1cm]   (1) edge (6);
 \path [bend right, distance=1cm]   (2) edge (7);
\draw (4) -- (7);
\draw (5) -- (6);
\draw (8) -- (10);

\end{tikzpicture}

\caption{A threshold graph where $D$ is the set of dominating vertices (forming a clique), $I$ is the set of independent vertices and the labels indicate the construction order of the graph. For the sake of simplicity, we omit all edges implied by the order of addition and type of the vertex added, except those of three diametral paths, namely $(1,6,5)$ and $(2,7,4)$ of Type II and $(3,10,8)$ of Type I.}
\label{fig:threshold}
\end{figure}

\begin{theorem}\label{thm:threshold}
Given a connected threshold graph $G$, $\#TDDP(G) \leq min\{|D|, |I|, \lfloor(|D| + |I|)/3\rfloor\}$.
\end{theorem}

\begin{proof}
We first observe that a connected threshold graph has diameter $2$ since the last vertex added must be a dominating vertex to ensure connectivity. 

Further, we observe that a diametral path of a connected threshold graph can be of two different types. A Type I diametral path is between two independent vertices through a single, dominating vertex. This occurs if the dominating vertex has a label larger than the two isolated vertices. See, for example, the path $(3,10,8)$ in Figure \ref{fig:threshold}. A Type II diametral path is between a dominating vertex and an isolated vertex through another dominating vertex. Such a diametral path occurs whenever the label of the isolated vertex falls between the labels of the two dominating vertices.  The paths $(1,6,5)$ and $(2,7,4)$ in Figure \ref{fig:threshold} are examples of Type II diametral paths.
As noted, every diametral path contains either a single dominating vertex or a single isolated vertex. Moreover, every diametral path includes three vertices. Thus, the upper bound follows.
\end{proof}

We note that although the bound in Theorem \ref{thm:threshold} holds with equality for the graph in Figure \ref{fig:threshold}, $\#TDDP(G)$ can be arbitrarily less than the bound. Consider a connected threshold graph $G$ constructed as $k$ dominating vertices followed by $k$ isolated vertices, ending with a single dominating vertex that connects the graph.  While the upper bound indicated would be $\lfloor (2k+1)/3 \rfloor$, we have $\#TDDP(G) = 1$. 

A connected threshold graph $G$ has diameter $2$ (Theorem~\ref{thm:threshold}),
and every diametral path $(a,b,c)$ is a path on three vertices satisfying
$a<c<b$ with either Type~I: $\{a,c\}\in I$ and $b\in D$; or Type~II:
$\{a,b\}\in D$ and $c\in I$. In both cases, the center $b$ is a dominating
vertex. Below, we describe a dynamic program that computes $\#TDDP(G)$ for threshold graphs.

\paragraph{The dynamic program \texttt{DP\_threshold}}

We fix the construction order of $G$ and label the vertices $1,2,\dots,n$
accordingly, so that vertex $t$ is the $t$-th vertex added; each vertex is of
type $D$ or $I$ according to whether it was added as a dominating or an isolated
vertex.  We exploit the fact that for
any vertex $v$, all earlier vertices of a given type bear the same adjacency
relation to $v$, so the algorithm need only track \emph{how many} unused vertices
of each type precede the current one, not which.

We sweep the vertices in order $1,\dots,n$, maintaining the state $(i,\,\delta,\,m)$,
where, after processing the prefix $\{1,\dots,t\}$:
\begin{itemize}
  \item $i$ = number of independent vertices in the prefix that are \emph{unused}
        (not yet assigned to any chosen path);
  \item $\delta$ = number of dominating vertices in the prefix that are unused;
  \item $m$ = number of Type~II paths whose endpoint $a\in D$ and middle vertex
        $c\in I$ have been chosen (with $a<c\le t$) but whose center $b\in D$
        (with $b>c$) is not yet placed --- the \emph{pending} Type~II paths.
\end{itemize}
Let $F_t(i,\delta,m)$ be the maximum number of \emph{completed} diametral paths
achievable over the prefix $\{1,\dots,t\}$ ending in state $(i,\delta,m)$. The
initial state is $F_0(0,0,0)=0$, all others $-\infty$. The transition from $t$ to $t+1$ depends on the type of vertex $v=t+1$.

\medskip
\noindent\emph{If $v\in I$}, one of:
\begin{itemize}
  \item leave unused: $(i,\delta,m)\to(i{+}1,\delta,m)$, completed $+0$;
  \item open a Type~II (use $v$ as a middle $c$, pairing it with an unused
        earlier dominating vertex as the endpoint $a$): requires $\delta\ge1$;
        $(i,\delta,m)\to(i,\delta{-}1,m{+}1)$, completed $+0$;
\end{itemize}

\noindent\emph{If $v\in D$}, one of:
\begin{itemize}
  \item leave unused: $(i,\delta,m)\to(i,\delta{+}1,m)$, completed $+0$;
  \item use as a Type~I center $b$ (consuming two unused independent
        vertices, necessarily earlier than $v$): requires $i\ge2$;
        $(i,\delta,m)\to(i{-}2,\delta,m)$, completed $+1$;
  \item use as the center $b$ completing a pending Type~II: requires
        $m\ge1$; $(i,\delta,m)\to(i,\delta,m{-}1)$, completed $+1$;
\end{itemize}

Note that when a vertex $v$ of type $I$ is left unused, it remains in the unused pool $i$ and may be consumed later by the center of a Type-I path.  Similarly, when a vertex $v$ of type $D$ is left unused, it remains in the unused pool $\delta$ and may be consumed later when an independent middle vertex opens a potential Type~II path.  The number of paths completed at each step are used to update $F$.

At the end, any pending Type~II paths cannot be completed and are discarded, yielding:
\[
  \#TDDP(G) \;=\; \max_{i,\delta}\; F_n(i,\delta,0).
\]

\begin{theorem}\label{thm:thr-dp}
For a connected threshold graph $G$ on $n$ vertices, \texttt{DP\_threshold}
computes $\#TDDP(G)$ correctly in time $O(n^3)$.
\end{theorem}

\begin{proof}
\emph{Validity (every solution of \texttt{DP\_threshold} is a set of totally disjoint diametral paths).}
Each vertex is consumed by at most one transition: an independent vertex is
either left in $i$, removed from $i$ as a Type~I endpoint when a center fires, or
removed as a middle $c$ when it opens a Type~II; a dominating vertex is either
left in $\delta$, removed from $\delta$ as a Type~II endpoint $a$, or used as the
current vertex serving as the center $b$ of a Type~I or Type~II path. Since each
completion removes its vertices from the pools (or is the current vertex, used
once), no vertex is shared between two completed paths, so the paths are totally
disjoint. Each completed path is a genuine diametral path: a Type~I center $b$
consumes two earlier independent vertices $a,c<b$, non-adjacent to each other and
both adjacent to $b$; a Type~II completion at $v=b$ uses a pending pair $(a,c)$
with $a<c<b$, which is a Type~II diametral path by the characterization above.

\emph{Optimality (every set of totally disjoint diametral paths is realized by some
\texttt{DP\_threshold} run).}
Let $\mathcal{S}$ be a family of pairwise totally disjoint diametral paths. Since
the paths of $\mathcal{S}$ are vertex-disjoint, each vertex plays exactly one role
in $\mathcal{S}$. Process the vertices in order and, for each $v$, fire the
transition matching its role: a Type~I center $b$ fires the Type~I transition
(consuming its two endpoints $a,c$, both earlier than $v$ and present in the pool
$i$); the center $b$ of a Type~II path fires the completion (its pending entry was
created at its middle $c$); the middle $c$ of a Type~II path fires the open
transition (consuming its endpoint $a$ from $\delta$, present since $a<c$ and
unused); Type~I endpoints and Type~II endpoints $a$ are left in their pools until
their firing step; all other vertices are left unused. The only resource demands
are two unused independent vertices at a Type~I firing and one unused dominating
vertex at a Type~II opening. Both are met: the required vertices belong to
$\mathcal{S}$, are earlier than the firing vertex, and play no other role, so they
sit unused in the pools. Because all earlier vertices of a type relate identically
to the current vertex, the pools need not distinguish their identities --- any two
earlier unused independent vertices serve as the Type~I endpoints $a,c$, and any
earlier unused dominating vertex serves as the Type~II endpoint $a$. Hence the run
is valid and completes exactly $|\mathcal{S}|$ paths, giving
$\max_{i,\delta}F_n(i,\delta,0)\ge\#TDDP(G)$; with validity, equality holds.

\emph{Complexity.}
The state $(i,\delta,m)$ has $O(n)$ values per coordinate, hence $O(n^3)$ states,
each with $O(1)$ transitions, for $O(n^3)$ time and $O(n^2)$ space (two sweep
layers). The construction order is computable in linear time~\cite{threshold}.
\end{proof}

Below is how \texttt{DP\_threshold} applies to the example in Figure \ref{fig:threshold}.

\begin{table}[htb]
\centering
\begin{tabular}{c c c l c c}
\toprule
step & $v$ & type & action & $(i,\delta,m)$ & paths completed in $F$ \\
\midrule
0  & --- & --- & initialize                                   & $(0,0,0)$ & 0 \\
1  & 1   & $D$ & leave unused                                 & $(0,1,0)$ & 0 \\
2  & 2   & $D$ & leave unused                                 & $(0,2,0)$ & 0 \\
3  & 3   & $I$ & leave unused                                 & $(1,2,0)$ & 0 \\
4  & 4   & $I$ & open Type~II ($a=2$, middle $c=4$)           & $(1,1,1)$ & 0 \\
5  & 5   & $I$ & open Type~II ($a=1$, middle $c=5$)           & $(1,0,2)$ & 0 \\
6  & 6   & $D$ & complete pending Type~II (center $b=6$)      & $(1,0,1)$ & 1 \\
7  & 7   & $D$ & complete pending Type~II (center $b=7$)      & $(1,0,0)$ & 2 \\
8  & 8   & $I$ & leave unused                                 & $(2,0,0)$ & 2 \\
9  & 9   & $I$ & leave unused                                 & $(3,0,0)$ & 2 \\
10 & 10  & $D$ & Type~I center ($b=10$, consumes $i=3,8$)     & $(1,0,0)$ & 3 \\
\bottomrule
\end{tabular}
\caption{Trace of \texttt{DP\_threshold} on the threshold graph of Figure~\ref{fig:threshold},
recovering the three totally disjoint diametral paths $(1,6,5)$, $(2,7,4)$, and $(3,10,8)$.
The state is $(i,\delta,m)$ and the last column is the number of completed paths.}
\label{tab:thr-trace}
\end{table}

\subsection{Proper Interval Graphs}

Proper interval graphs are a structured subclass on which diametral paths admit a
clean description, yet computing $\#TDDP$ appears to resist the disjoint-path
techniques that succeed for the classes above. In this section, we describe the
structure of diametral paths in proper interval graphs, give a tight upper bound
on $\#TDDP$, and explain the obstruction that leaves the exact computation open.

\textit{Interval graphs} are the intersection graphs of overlapping intervals on
the number line. It is well known that the maximal cliques of an interval graph
can be linearly ordered so that every vertex is contained in a sequence of
consecutive maximal cliques \cite{interval}. \textit{Proper interval graphs} are
the subclass admitting an interval representation in which no interval is totally
contained within another; equivalently, the orderings of the intervals by left and by right endpoint coincide. Given a proper interval graph on $n$ vertices, we
number the vertices $1$ to $n$ in left to right order of their associated
intervals, and use vertices and intervals interchangeably with their numbers. For
a vertex $v$, let $\ell(v)$ and $r(v)$ denote the smallest and largest indices of
a maximal clique containing $v$ in the linear ordering of maximal cliques; both
are non-decreasing in the vertex number, and two vertices are adjacent if and only
if their clique-index ranges overlap. We write $D$ for the diameter of $G$.

Figure \ref{fig:PIG} shows a proper interval graph with its interval
representation and the linear ordering of its maximal cliques. The two paths
$(1,4,8,11,13)$ and $(2,3,6,10,12)$ shown in bold are diametral and totally
disjoint, so this graph has at least (and exactly) two totally disjoint diametral paths.

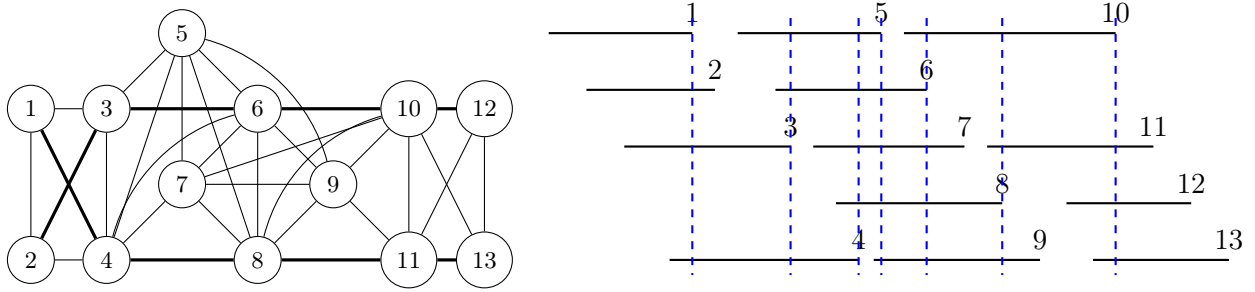
\begin{figure}[htb]
    \centering
\begin{tikzpicture}

% First figure on the left
    \begin{scope}[shift={(-4,0)}] % Shift this scope to the left
% Nodes

\node[circle, draw, fill=white] (1) at (0, 3) {\footnotesize 1};
\node[circle, draw, fill=white] (2) at (0, 1) {\footnotesize 2};
\node[circle, draw, fill=white] (3) at (1, 3) {\footnotesize 3};
\node[circle, draw, fill=white] (4) at (1, 1) {\footnotesize 4};
\node[circle, draw, fill=white] (5) at (2, 4) {\footnotesize 5};
\node[circle, draw, fill=white] (6) at (3, 3) {\footnotesize 6};
\node[circle, draw, fill=white] (7) at (2, 2) {\footnotesize 7};
\node[circle, draw, fill=white] (8) at (3, 1) {\footnotesize 8};
\node[circle, draw, fill=white] (9) at (4, 2) {\footnotesize 9};
\node[circle, draw, fill=white] (10) at (5, 3) {\footnotesize 10};
\node[circle, draw, fill=white] (11) at (5, 1) {\footnotesize 11};
\node[circle, draw, fill=white] (12) at (6, 3) {\footnotesize 12};
\node[circle, draw, fill=white] (13) at (6, 1) {\footnotesize 13};

% Edges
\draw (1) -- (2);
\draw (1) -- (3);
\draw[very thick] (1) -- (4);
\draw[very thick] (2) -- (3);
\draw (2) -- (4);
\draw (3) -- (4);
\draw (3) -- (5);
\draw[very thick] (3) -- (6);
\draw (4) -- (5);
%\draw (3) -- (5);
%\draw (1,1.31) .. controls (2.5, -1) and (0.5, 1.5) ..  (5);
 \path [bend left, distance=0.8cm]   (4) edge (6);
\draw (4) -- (7);
\draw[very thick] (4) -- (8);
\draw (5) -- (6);
\draw (5) -- (7);
\draw (5) -- (8);
 \path [bend left, distance=0.8cm]   (5) edge (9);
\draw (6) -- (7);
\draw (6) -- (8);
\draw (6) -- (9);
\draw[very thick] (6) -- (10);
\draw (7) -- (8);
\draw (7) -- (9);
\draw (7) -- (10);
\draw (8) -- (9);
 \path [bend left, distance=0.8cm]   (8) edge (10);
\draw[very thick] (8) -- (11);
\draw (9) -- (10);
\draw (9) -- (11);
\draw (10) -- (11);
\draw (10) -- (13);
\draw[very thick] (10) -- (12);
\draw[very thick] (11) -- (13);
\draw (11) -- (12);
\draw (13) -- (12);

\end{scope}

% First figure on the right
    \begin{scope}[shift={(2.85,0)}] % Shift this scope to the right

% Intervals

\draw[thick] (0, 4) -- (1.9, 4) node[above] {1};
\draw[thick] (0.5, 3.25) -- (2.2, 3.25) node[above] {2};
\draw[thick] (1, 2.5) -- (3.2, 2.5) node[above] {3};
\draw[thick] (1.6, 1) -- (4.1, 1) node[above] {4};
\draw[thick] (2.5, 4) -- (4.4, 4) node[above] {5};
\draw[thick] (3, 3.25) -- (5, 3.25) node[above] {6};
\draw[thick] (3.5, 2.5) -- (5.5, 2.5) node[above] {7};
\draw[thick] (3.8, 1.75) -- (6, 1.75) node[above] {8};
\draw[thick] (4.3, 1) -- (6.5, 1) node[above] {9};
\draw[thick] (4.7, 4) -- (7.5, 4) node[above] {10};
\draw[thick] (5.8, 2.5) -- (8, 2.5) node[above] {11};
\draw[thick] (6.85, 1.75) -- (8.5, 1.75) node[above] {12};
\draw[thick] (7.2, 1) -- (9, 1) node[above] {13};

%Maximal cliques
%\node at (1.6, 2.1) {\textcolor{blue}{$C^1$}};
\draw (1.9, 4.2) [blue, thick, dashed] -- (1.9, 0.8);
\draw (3.2, 4.2) [blue, thick, dashed] -- (3.2, 0.8); 
\draw (4.1, 4.2) [blue, thick, dashed] -- (4.1, 0.8); 
\draw (4.4, 4.2) [blue, thick, dashed] -- (4.4, 0.8); 
\draw (5, 4.2) [blue, thick, dashed] -- (5, 0.8); 
\draw (6, 4.2) [blue, thick, dashed] -- (6, 0.8); 
\draw (7.5, 4.2) [blue, thick, dashed] -- (7.5, 0.8); 
\end{scope}

\end{tikzpicture}

\caption{A proper interval graph and its interval representation where intervals are linearly ordered. Dashed vertical lines indicate the maximal cliques. }
\label{fig:PIG}
\end{figure}

We first remark that a natural greedy algorithm does not compute $\#TDDP$. The graph in Figure~\ref{fig:longpath} has
diameter $3$. Starting from the left, a greedy algorithm that repeatedly extends a path by its highest-numbered available neighbor constructs $p_1=(1,4,6,9)$;
forced onto lower-numbered vertices to stay disjoint, the next path it builds is
$p_2=(2,3,5,7,8)$, which has length $4$ and is therefore not diametral.

Next, we investigate the structure of the diametral paths and note that they are monotone in the vertex order.

\begin{lemma}\label{lem:pi-monotone}
Let $u<v$ be non-adjacent vertices of $G$, and let $P=(x_0,x_1,\dots,x_t)$ be a
shortest path from $u=x_0$ to $v=x_t$. Then $x_0<x_1<\cdots<x_t$. In particular,
every diametral path of $G$, oriented from its smaller to its larger endpoint, is
strictly increasing in the vertex order.
\end{lemma}

\begin{proof}
Suppose some step satisfies $r(x_{j+1})\le r(x_j)$. Since $(x_{j-1}, x_j)\in E$, their
clique ranges overlap, so $\ell(x_j)\le r(x_{j-1})$; combined with
$r(x_{j+1})\le r(x_j)$ and the absence of proper containment, the range of
$x_{j+1}$ overlaps that of $x_{j-1}$, hence $(x_{j-1},x_{j+1})\in E$. Replacing
$x_{j-1},x_j,x_{j+1}$ by $x_{j-1},x_{j+1}$ yields a shorter path, contradicting
that $P$ is shortest. Thus $r(\cdot)$ strictly increases along $P$, and since $r$
is nondecreasing in the vertex number, so is the vertex label.
\end{proof}

The level structure of Lemma~\ref{lem:pi-monotone} might suggest assigning each vertex a single position and packing diametral paths through a layered network. This fails, because a vertex can occupy different positions on different diametral paths. In the graph of Figure~\ref{fig:longpath} (diameter 3), the diametral path $(2,4,6,9)$ uses vertex 2 as its left endpoint, at position 0, while the diametral path $(1,2,5,8)$ uses the same vertex 2 as its second vertex, at position 1. No assignment of one fixed level per vertex is therefore consistent with all diametral paths.

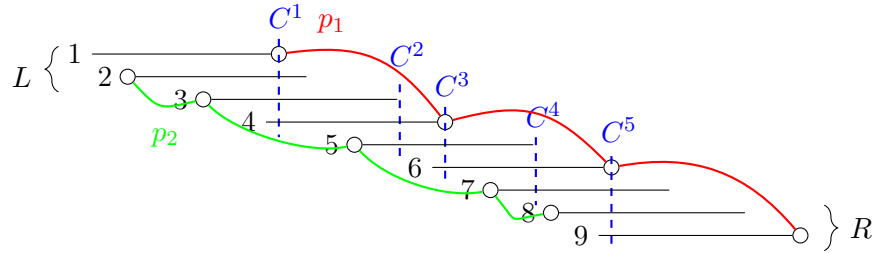
\begin{figure}[htb]
    \centering
    \begin{tikzpicture}
    % Path labels
    \node at (2.2,2) {\textcolor{red}{$p_1$}};

    % Nodes
    \node (p1l) at (-1.2,1.6) {1};
    \node [draw, circle, inner sep=2pt] (p1r) at (1.5,1.6) {};

    \node at (-0.8,1.3) {2};
    \node [draw, circle, inner sep=2pt] (p2l) at (-0.5,1.3) {};
    \node (p2r) at (2,1.3) {};

    \node at (0.2,1) {3};
    \node [draw, circle, inner sep=2pt]  (p3l) at (0.5,1) {};
    \node (p3r) at (3.2,1) {};

    \node (p4l) at (1.1,0.7) {4};
    \node[draw, circle, inner sep=2pt] (p4r) at (3.7,0.7) {};

    \node at (2.2,0.4) {5};
    \node [draw, circle, inner sep=2pt] (p5l) at (2.5,0.4) {};
    \node (p5r) at (5,0.4) {};

     \node (p6l) at (3.3,0.1) {6};
     \node[draw, circle, inner sep=2pt] (p6r) at (5.9,0.1) {};

      \node at (4,-0.2) {7};
     \node [draw, circle, inner sep=2pt] (p7l) at (4.3,-0.2) {};
     \node (p7r) at (6.8,-0.2) {};

     \node at (4.8,-0.5) {8};
     \node [draw, circle, inner sep=2pt] (p8l) at (5.1,-0.5) {};
     \node (p8r) at (7.8,-0.5) {};

     \node (p9l) at (5.5,-0.8) {9};
     \node[draw, circle, inner sep=2pt] (p9r) at (8.4,-0.8) {};

   % L and R
    \draw [decorate, decoration={brace,  amplitude=5pt,mirror, raise=1.3ex}]
  (-1.2,1.7) -- (-1.2, 1.1) node[midway,black, yshift=2em]{};
   \node at (-1.9,1.3) {$L$};

   \draw [decorate, decoration={brace, mirror, amplitude=5pt, raise=1.3ex}]
  (8.5,-1) -- (8.5,-0.4) node[midway,black, yshift=2em]{};
   \node at (9.2,-0.7) {$R$};

%vertical lines for maximal cliques
\node at (1.6, 2.1) {\textcolor{blue}{$C^1$}};
\draw (1.5,1.8) [blue, thick, dashed] -- (1.5,0.5);

\node at (3.2, 1.6) {\textcolor{blue}{$C^2$}};
\draw (3.1,1.2) [blue, thick, dashed] -- (3.1,0.2);

\node at (3.8,1.2) {\textcolor{blue}{$C^3$}};
\draw (3.7,0.9) [blue, thick, dashed] -- (3.7,-0.1);

\node at (5,0.8) {\textcolor{blue}{$C^4$}};
\draw (4.9,0.5) [blue, thick, dashed] -- (4.9,-0.4);

\node at (6,0.6) {\textcolor{blue}{$C^5$}};
\draw (5.9,0.25) [blue, thick, dashed] -- (5.9,-1);

    % Horizontal Edges
    \draw (p1l) -- (p1r);
    \draw (p2l) -- (p2r);
    \draw (p3l) -- (p3r);
    \draw (p4l) -- (p4r);
    \draw (p5l) -- (p5r);
    \draw (p6l) -- (p6r);
    \draw (p7l) -- (p7r);
    \draw (p8l) -- (p8r);
    \draw (p9l) -- (p9r);

    % Paths
    %p1
    \path [red,bend left, thick, distance=1cm]   (p1r) edge (p4r);
    \path [red,bend left, thick, distance=1cm]   (p4r) edge (p6r);
    \path [red,bend left, thick, distance=1cm]   (p6r) edge (p9r);
   
   %p2: left 2,3,5,7,8
   \path [green, bend right, thick, distance=0.5cm]   (p2l) edge (p3l);
   \path [green, bend right, thick, distance=0.5cm]   (p3l) edge (p5l);
   \path [green, bend right, thick, distance=0.5cm]   (p5l) edge (p7l);
   \path [green, bend right, thick, distance=0.5cm]   (p7l) edge (p8l);

   \node at (0,0.5) {\textcolor{green}{$p_2$}};

\end{tikzpicture}

    \caption{A proper interval graph along with its interval representation. Intervals (and vertices) are linearly ordered and increasingly numbered from left to right. Maximal cliques $C^1$ to $C^5$ are also linearly ordered from left to right as shown by dashed vertical lines.
}
    \label{fig:longpath}
\end{figure}

This obstruction defeats the natural formulations. A maximum-flow model that splits each vertex and routes unit flow through layered copies miscounts, because a vertex's level is not well defined. The difficulty is structurally distinct from the classical vertex-disjoint shortest path problem between \emph{fixed} terminal pairs: here the endpoints range freely over all diametral pairs, and it is precisely this freedom that resists a layered packing.

Nonetheless, the level structure yields a tight upper bound. For $0\le j\le D$,
let
\[
  \Lambda_j \;=\; \{\, v\in V : v \text{ is the $j$-th vertex of some diametral
  path}\,\}
\]
be the set of vertices that can occupy position $j$.

\begin{theorem}\label{thm:pi-bound}
For a proper interval graph $G$ with diameter $D$,
\[
  \#TDDP(G) \;\le\; \min_{0\le j\le D} |\Lambda_j|.
\]
\end{theorem}

\begin{proof}
Let $\mathcal{S}$ be a set of pairwise totally disjoint diametral paths, each
oriented increasingly by Lemma \ref{lem:pi-monotone}. For a fixed $j$, the $j$-th
vertices of the paths in $\mathcal{S}$ are pairwise distinct, as the paths are
vertex disjoint, and each lies in $\Lambda_j$. Hence
$|\mathcal{S}|\le|\Lambda_j|$; taking the minimum over $j$ gives the bound.
\end{proof}

The bound is constructive: a vertex $v$ lies in $\Lambda_j$ exactly when some diametral pair $(s,t)$ satisfies $d(s,v)=j$ and $d(v,t)=D-j$, which can be tested
from the all-pairs distance matrix; hence $\min_{0\le j\le D}|\Lambda_j|$ is computable in polynomial time.

This bound is the proper interval analogue of the centre bound for trees
(Theorem \ref{thm:tree}) and the barrier-edge bound for $2$-paths (Corollary
\ref{cor:2path}): in each case a position that every diametral path must
traverse limits the number of disjoint such paths. The bound is not always
attained, as the multi-position phenomenon of Figure \ref{fig:longpath} can force
$\#TDDP(G)$ strictly below $\min_j|\Lambda_j|$. Whether $\#TDDP$ can be computed
in polynomial time on proper interval graphs remains open.

%% file: 04-ExtremalGraphs.tex
 In this section, we are interested in the construction of extremal graphs that answer the following question: for given $k$ and $d$, what is the minimum number of edges in a graph having $k$ totally disjoint diametral paths of length $d$? This minimum number of edges is denoted by $MinE(k,d)$. Graphs having $k$ totally disjoint diametral paths of length $d$ and having exactly $MinE(k,d)$ edges are called \textit{extremal}. We find $MinE(k,d)$ for all $k\geq 1$ and $d\geq 1$, identifying extremal graphs for each parameter pair $k$ and $d$.

We first define a class of graphs, $DDPG_{k,d}$, each having $k$ totally disjoint diametral paths of length $d$, for $d \geq 4$.  $DDPG_{k,d}$ is defined to be $k$ vertex disjoint paths of length $d$, $p_1, \ldots,p_k$, where each path $p_i$ has vertices $p_{i,0}, \ldots , p_{i,d}$, connected by edges in that order.  Let $h = \lfloor d/2\rfloor $ and $h' =\lceil d/2 \rceil$.  Each path $p_i$, $1 < i \leq k$, is connected to path $p_1$ by two edges $(p_{1,h}, p_{i,(h-1)})$ and $(p_{1,h'}, p_{i,(h'+1)})$.  Figure \ref{fig:DDPG34} and \ref{fig:DDPG35} show graphs $DDPG_{3,4}$ and $DDPG_{3,5}$, respectively.

\begin{figure}[htb]
    \centering
    \begin{tikzpicture}
    % Path labels
    \node at (-1,1.6) {$p_1$};
    \node at (-1,0) {$p_2$};
    \node at (-1,-1.6) {$p_3$};

    % Nodes
    \node[draw, circle, inner sep=2pt] (p10) at (0,1.6) {};
    \node[above=2pt] at (p10.north) {$p_{1,0}$};

    \node[draw, circle, inner sep=2pt] (p11) at (2,1.6) {};
    \node[above=2pt] at (p11.north) {$p_{1,1}$};

    \node[draw, circle, inner sep=2pt] (p12) at (4,1.6) {};
    \node[above=2pt] at (p12.north) {$p_{1,2}$};

    \node[draw, circle, inner sep=2pt] (p13) at (6,1.6) {};
    \node[above=2pt] at (p13.north) {$p_{1,3}$};

    \node[draw, circle, inner sep=2pt] (p14) at (8,1.6) {};
    \node[above=2pt] at (p14.north) {$p_{1,4}$};

    \node[draw, circle, inner sep=2pt] (p20) at (0,0) {};
    \node[above=2pt] at (p20.north) {$p_{2,0}$};

    \node[draw, circle, inner sep=2pt] (p21) at (2,0) {};
    \node[above=2pt] at (p21.north) {$p_{2,1}$};

    \node[draw, circle, inner sep=2pt] (p22) at (4,0) {};
    \node[above=2pt] at (p22.north) {$p_{2,2}$};

    \node[draw, circle, inner sep=2pt] (p23) at (6,0) {};
    \node[above=2pt] at (p23.north) {$p_{2,3}$};

    \node[draw, circle, inner sep=2pt] (p24) at (8,0) {};
    \node[above=2pt] at (p24.north) {$p_{2,4}$};

    \node[draw, circle, inner sep=2pt] (p30) at (0,-1.6) {};
    \node[above=2pt] at (p30.north) {$p_{3,0}$};

    \node[draw, circle, inner sep=2pt] (p31) at (2,-1.6) {};
    \node[above=2pt] at (p31.north) {$p_{3,1}$};

    \node[draw, circle, inner sep=2pt] (p32) at (4,-1.6) {};
    \node[above=2pt] at (p32.north) {$p_{3,2}$};

    \node[draw, circle, inner sep=2pt] (p33) at (6,-1.6) {};
    \node[above=2pt] at (p33.north) {$p_{3,3}$};

    \node[draw, circle, inner sep=2pt] (p34) at (8,-1.6) {};
    \node[above=2pt] at (p34.north) {$p_{3,4}$};

    % Horizontal Edges
    \draw (p10) -- (p11) -- (p12) -- (p13) -- (p14);
    \draw (p20) -- (p21) -- (p22) -- (p23) -- (p24);
    \draw (p30) -- (p31) -- (p32) -- (p33) -- (p34);

    % Vertical and Diagonal Edges
    
    \draw (p12) -- (p21);
    \draw (p12) -- (p23);
    \draw (p12) -- (p31);
    \draw (p12) -- (p33);

\end{tikzpicture}
    \caption{$DDPG_{3,4}$ where $h=h'=2$.}
        \label{fig:DDPG34}
\end{figure}
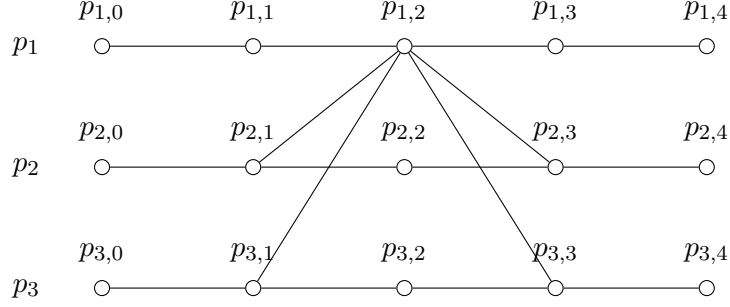

\begin{figure}[htb]
    \centering
    \begin{tikzpicture}
    % Path labels
    \node at (-1,1.6) {$p_1$};
    \node at (-1,0) {$p_2$};
    \node at (-1,-1.6) {$p_3$};

    % Nodes
    \node[draw, circle, inner sep=2pt] (p10) at (0,1.6) {};
    \node[above=2pt] at (p10.north) {$p_{1,0}$};

    \node[draw, circle, inner sep=2pt] (p11) at (2,1.6) {};
    \node[above=2pt] at (p11.north) {$p_{1,1}$};

    \node[draw, circle, inner sep=2pt] (p12) at (4,1.6) {};
    \node[above=2pt] at (p12.north) {$p_{1,2}$};

    \node[draw, circle, inner sep=2pt] (p13) at (6,1.6) {};
    \node[above=2pt] at (p13.north) {$p_{1,3}$};

    \node[draw, circle, inner sep=2pt] (p14) at (8,1.6) {};
    \node[above=2pt] at (p14.north) {$p_{1,4}$};

    \node[draw, circle, inner sep=2pt] (p15) at (10,1.6) {};
    \node[above=2pt] at (p15.north) {$p_{1,5}$};

    \node[draw, circle, inner sep=2pt] (p20) at (0,0) {};
    \node[above=2pt] at (p20.north) {$p_{2,0}$};

    \node[draw, circle, inner sep=2pt] (p21) at (2,0) {};
    \node[above=2pt] at (p21.north) {$p_{2,1}$};

    \node[draw, circle, inner sep=2pt] (p22) at (4,0) {};
    \node[above=2pt] at (p22.north) {$p_{2,2}$};

    \node[draw, circle, inner sep=2pt] (p23) at (6,0) {};
    \node[above=2pt] at (p23.north) {$p_{2,3}$};

    \node[draw, circle, inner sep=2pt] (p24) at (8,0) {};
    \node[above=2pt] at (p24.north) {$p_{2,4}$};

    \node[draw, circle, inner sep=2pt] (p25) at (10,0) {};
    \node[above=2pt] at (p25.north) {$p_{2,5}$};

    \node[draw, circle, inner sep=2pt] (p30) at (0,-1.6) {};
    \node[above=2pt] at (p30.north) {$p_{3,0}$};

    \node[draw, circle, inner sep=2pt] (p31) at (2,-1.6) {};
    \node[above=2pt] at (p31.north) {$p_{3,1}$};

    \node[draw, circle, inner sep=2pt] (p32) at (4,-1.6) {};
    \node[above=2pt] at (p32.north) {$p_{3,2}$};

    \node[draw, circle, inner sep=2pt] (p33) at (6,-1.6) {};
    \node[above=2pt] at (p33.north) {$p_{3,3}$};

    \node[draw, circle, inner sep=2pt] (p34) at (8,-1.6) {};
    \node[above=2pt] at (p34.north) {$p_{3,4}$};

    \node[draw, circle, inner sep=2pt] (p35) at (10,-1.6) {};
    \node[above=2pt] at (p35.north) {$p_{3,5}$};
    % Horizontal Edges
    \draw (p10) -- (p11) -- (p12) -- (p13) -- (p14) -- (p15) ;
    \draw (p20) -- (p21) -- (p22) -- (p23) -- (p24) -- (p25);
    \draw (p30) -- (p31) -- (p32) -- (p33) -- (p34) -- (p35);

    % Vertical and Diagonal Edges
    
    \draw (p12) -- (p21);
    \draw (p13) -- (p24);
    \draw (p12) -- (p31);
    \draw (p13) -- (p34);

\end{tikzpicture}
    \caption{$DDPG_{3,5}$ where $h=2$ and $h'=3$.}
    \label{fig:DDPG35}
\end{figure}
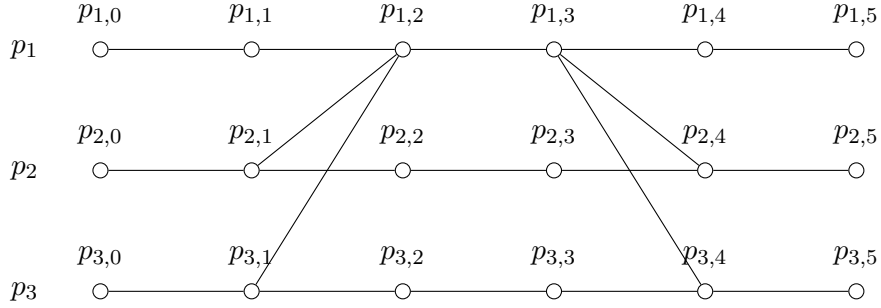

\begin{theorem} $DDPG_{k,d}$ has $k$ totally disjoint diametral paths of length $d$ for all  $d\geq 4$.
\end{theorem}
\begin{proof}
Paths $p_1, \ldots p_k$ are pairwise vertex disjoint and have length $d$.  The edges added, connecting other paths to $p_1$, do not reduce the distance between $p_{i,0}$ and $p_{i,d}$. We must show that the distance between every other pair of vertices in the graph is less than or equal to $d$, thereby maintaining  $p_1, \ldots, p_k$ as diametral paths.
  
First, any pair of vertices in the same path $p_i$ is at distance less than or equal to $d$.  Consider the situation when $d$ is even; in this case, $h$ and $h'$ are equal.  Vertices $p_{i,0}$ and $p_{j,d}$, for $1 \leq i < j \leq k$, are at distance $d$ by the path $(p_{i,0}, \ldots p_{i,(h-1)}, p_{1,h}, p_{j,(h+1)},\ldots p_{j,d})$, as $ (h - 1) + 1 +1 + (d - h - 1) = d$. Similarly, consider vertices $p_{i,0}$ and $p_{j,0}$, for $1 \leq  i < j \leq k$.  They are at distance $d$ by the path $(p_{i,0}, \ldots p_{i,(h-1)}, p_{1,h}, p_{j,(h-1)}, \ldots p_{j,0})$, as $(h-1) + 1 + 1 + (h - 1) = 2h = d$. Similar arguments hold for distances between $p_{i,d}$ and $p_{j,0}$ and between $p_{i,d}$ and $p_{j,d}$, for $1 \leq i < j \leq k$. Pairs of vertices within any of these paths are at distance less than $d$.  Vertices $p_{i,h}$ and $p_{j,h}$, for $1 < i \neq j \leq k$, are at distance $4$ by path $(p_{i,h}, p_{i,(h-1)}, p_{1,h}, p_{j,(h-1)}, p_{j,h})$.  This distance provides the lower bound on $d$ for the class $DDPG_{k,d}$.  A vertex $p_{i,h}$ is at distance $2$ from $p_{1,h}$ and, thus, at distance less than or equal to $h+2$ for any other vertex.

Now let $d$ be odd.  Note that $h' = h+1$ in this case.  Consider vertices $p_{m,0}$ and $p_{n,d}$ for $1 \leq m < n \leq k$.  They are at distance $d$ by the path $(p_{m,0}, \ldots, p_{m.(h-1)}, p_{1,h}, p_{1,h'}, p_{n,(h'+1)}, \ldots, p_{n,d})$, as $(h - 1) + 1 + 1 + 1+ (d - h - 2) = d$,  Consider vertices $p_{m,0}$ and $p_{n,0}$, for $1 \leq m < n \leq k$. They are at distance $d-1$ by the path $(p_{m,0}, \ldots p_{m,(h-1)}, p_{1,h}, \ldots p_{n,(h-1)},\ldots p_{n,0})$, as $h - 1 +1 +1 + h - 1 = 2h = d-1$.  Pairs of vertices within any of the paths are at distance less than $d$.  Vertices $p_{m,h}$ and $p_{m,(h+1)}$ are at distance $5$ from $p_{n,h}$ or $p_{n,(h+1)}$ by choosing the correct neighbor, for $1< m < n \leq k$. All other distances between pairs of vertices are less than or equal to $d$, based on arguments similar to those for the even cases above, leaving our $k$ initial paths as disjoint diametral paths. 
\end{proof}

\begin{theorem} $DDPG_{k,d}$ is an extremal graph having a set of $k$ totally disjoint diametral paths of length $d$, for all  $d\geq 4$.
\end{theorem}

\begin{proof} We show that in an extremal graph every disjoint diametral path must be connected to other paths by at least two edges.  If only a single edge connects a path $p$ to a vertex $v$ outside $p$, then one end vertex of $p$ is at distance at least $\lceil d/2 \rceil$ from $v$.  If $v$ is part of another diametral path $r$, then the minimum possible maximum distance from an end of $p$ to an end of $r$ is $2 \lceil d/2 \rceil +1$, which is greater than $d$.  If $v$ is not part of another diametral path, then vertex $v$ must be connected to every other diametral path by at least two edges; otherwise, there would be two vertices at distance  $2 \lceil d/2 \rceil + 2$. However, having such a vertex $v$ results in a graph with at least $kd + 2(k-1) + 1$ edges, which is one more than graph $DDPG_{k,d}$. Thus, every vertex of an extremal graph must be contained in a diametral path, implying that $DDPG_{k,d}$ is extremal.
\end{proof}

\begin{corollary} $MinE(k,d) = kd+2(k-1)$ for $d \geq 4$ and $k>0$.
\end{corollary}

\begin{proof}
    The graph $DDPG_{k,d }$ above has $kd + 2(k-1)$ edges. The extremal paths together have $kd$ edges, and $k-1$ of these paths have $2$ edges connecting to path $p_1$.
\end{proof}

There remain the problems for distances $d<4$.  For $d=1$, a complete graph on $2k$ vertices is necessary, as every vertex must be a neighbor of every other vertex if the graph is to have diameter $1$.  So, $MinE(k,1) = 2k^2 - k$, for all $k>0$.

	For $d=2$, consider a construction similar to $DDPG_{k,d}$ given above,  with $k$ paths $p_1, \ldots, p_k$ of length $2$, connect the middle vertex of path $p_1$, i.e., $p_{1,1}$, to all three vertices in every other path $p_i, 1 < i \leq k$.  See Figure \ref{fig:DDPG32} for an example with $k = 3$.  Since it is similar in construction, we will call this class of graphs $DDPG_{k,2}$.  Each $DDPG_{k,2}$ has diameter $2$ and $k$ totally disjoint diametral paths.  The graph has $5k - 3$ edges; $3(k-1)$ edges to connect paths $p_i$, $1 < i \leq k$, to the path $p_1$ and $2k$ edges in the $k$ disjoint diametral paths $p_1, \ldots ,p_k$.

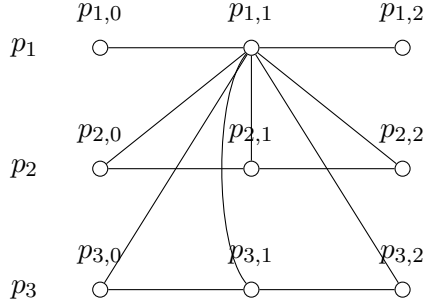
\begin{figure}[htb]
    \centering
    \begin{tikzpicture}
    % Path labels
    \node at (-1,1.6) {$p_1$};
    \node at (-1,0) {$p_2$};
    \node at (-1,-1.6) {$p_3$};

    % Nodes
    \node[draw, circle, inner sep=2pt] (p10) at (0,1.6) {};
    \node[above=2pt] at (p10.north) {$p_{1,0}$};

    \node[draw, circle, inner sep=2pt] (p11) at (2,1.6) {};
    \node[above=2pt] at (p11.north) {$p_{1,1}$};

    \node[draw, circle, inner sep=2pt] (p12) at (4,1.6) {};
    \node[above=2pt] at (p12.north) {$p_{1,2}$};

    \node[draw, circle, inner sep=2pt] (p20) at (0,0) {};
    \node[above=2pt] at (p20.north) {$p_{2,0}$};

    \node[draw, circle, inner sep=2pt] (p21) at (2,0) {};
    \node[above=2pt] at (p21.north) {$p_{2,1}$};

    \node[draw, circle, inner sep=2pt] (p22) at (4,0) {};
    \node[above=2pt] at (p22.north) {$p_{2,2}$};

    \node[draw, circle, inner sep=2pt] (p30) at (0,-1.6) {};
    \node[above=2pt] at (p30.north) {$p_{3,0}$};

    \node[draw, circle, inner sep=2pt] (p31) at (2,-1.6) {};
    \node[above=2pt] at (p31.north) {$p_{3,1}$};

    \node[draw, circle, inner sep=2pt] (p32) at (4,-1.6) {};
    \node[above=2pt] at (p32.north) {$p_{3,2}$};

    % Horizontal Edges
    \draw (p10) -- (p11) -- (p12);
    \draw (p20) -- (p21) -- (p22);
    \draw (p30) -- (p31) -- (p32);

    % Vertical and Diagonal Edges
    
    \draw (p11) -- (p20);
    \draw (p11) -- (p21);
    \draw (p11) -- (p22);
    \draw (p11) -- (p30); %this should be a curve
    \draw (p11) .. controls (1.5, 1) and (1.5, -1) ..  (p31);
    \draw (p11) -- (p32);

\end{tikzpicture}
    \caption{$DDPG_{3,2}$.}
        \label{fig:DDPG32}
\end{figure}

\begin{theorem}
$DDPG_{k,2}$ is an extremal graph having a set of $k$ totally disjoint diametral paths of length $2$.
\end{theorem}

\begin{proof}
 The given paths $p_1, \ldots , p_k$ are totally disjoint diametral paths of length $2$.  Every pair of vertices are either of the same given path or are at distance at most $2$ through vertex $p_{1,1}$.
 
 The graphs are minimal, as the removal of any edge from an original vertex disjoint path reduces the maximum number of totally disjoint paths from $k$ to $k-1$; removal of an edge incident to $p_{1,1}$ from vertex $v$ in another path $p_i$ for  $2\leq i \leq k$ makes $v$ and $p_{1,0}$ be at a distance of 3.  

Furthermore, $DDPG_{k,2}$ has a minimum number of edges for every $k$ since every vertex disjoint diametral path has $2$ edges connecting the three vertices in the path.  For a vertex $u$ of degree $1$, every other vertex must be a neighbor of its unique neighbor $v$.  Thus, every vertex of each vertex disjoint diametral path must be a neighbor of $v$.  There can be at most $2$ vertices of degree $1$, both in the same path, which $DDPG_{k,2}$ has.   
\end{proof}
\begin{corollary}
 $ MinE(k,2) =  5k-3.$
\end{corollary}
\begin{proof}
$DDPG_{k,2}$ has $5k-3$ edges, i.e., 3 edges to connect each path $p_i$, $1 < i \leq k$, to path $p_1$ and $2$ edges in each of the $k$ vertex disjoint diametral paths.  All these edges are necessary, as discussed in the proof above.   
\end{proof}

	For $k=3$, there are two interesting graph classes.  First define class $C_{k,3}$ to be $k$ vertex disjoint paths of length $3$, with the central two vertices of these paths connected into a complete graph $K_{2k}$.  Every vertex of degree $1$ is at distance $3$ from every other vertex of degree $1$ and at most distance $2$ from vertices of the $K_2$; all other vertices are direct neighbors.  These graphs are minimal, as removal of any edge disconnects the graph or results in a diametral pair at distance $4$.  A $C_{k,3}$ has $2k^2+k$ edges, i.e., $2k^2 - k$ edges in the complete graph and $2k$ edges connecting the $2k$ degree $1$ vertices to the vertices of the complete graph.

	Next, consider the following graph class $DDPG_{k,3}$, constructed similarly to the $DDPG_{k,d}$ graphs for $d>3$.  Given $k$ vertex disjoint paths $p_1,\ldots , p_k$ of length $3$, connect all vertices of paths $p_i,$ for all $2 \leq i \leq k$, to $p_1$ by edges $(p_{1,1}, p_{i,0}), (p_{1,1}, p_{i,1}), (p_{1,2}, p_{i,2}), (p_{1,2}, p_{i,3})$, excluding edge $(p_{1,1}, p_{2,1})$.  A $DDPG_{k,3}$ has $7k-5$ edges, i.e. $3k$ edges in the $k$ disjoint diametral paths and $4(k-1)$ edges connecting other paths to path $p_1$, less 1 by excluding edge $(p_{1,1}, p_{2,1})$.  See Figure \ref{fig:DDPG33} showing $DDPG_{3,3}$.

\begin{figure}[htb]
    \centering
    \begin{tikzpicture}
    % Path labels
    \node at (-1,1.6) {$p_1$};
    \node at (-1,0) {$p_2$};
    \node at (-1,-1.6) {$p_3$};

    % Nodes
    \node[draw, circle, inner sep=2pt] (p10) at (0,1.6) {};
    \node[above=2pt] at (p10.north) {$p_{1,0}$};

    \node[draw, circle, inner sep=2pt] (p11) at (2,1.6) {};
    \node[above=2pt] at (p11.north) {$p_{1,1}$};

    \node[draw, circle, inner sep=2pt] (p12) at (4,1.6) {};
    \node[above=2pt] at (p12.north) {$p_{1,2}$};

    \node[draw, circle, inner sep=2pt] (p13) at (6,1.6) {};
    \node[above=2pt] at (p13.north) {$p_{1,3}$};
   
    \node[draw, circle, inner sep=2pt] (p20) at (0,0) {};
    \node[above=2pt] at (p20.north) {$p_{2,0}$};

    \node[draw, circle, inner sep=2pt] (p21) at (2,0) {};
    \node[above=2pt] at (p21.north) {$p_{2,1}$};

    \node[draw, circle, inner sep=2pt] (p22) at (4,0) {};
    \node[above=2pt] at (p22.north) {$p_{2,2}$};

    \node[draw, circle, inner sep=2pt] (p23) at (6,0) {};
    \node[above=2pt] at (p23.north) {$p_{2,3}$};

    \node[draw, circle, inner sep=2pt] (p30) at (0,-1.6) {};
    \node[above=2pt] at (p30.north) {$p_{3,0}$};

    \node[draw, circle, inner sep=2pt] (p31) at (2,-1.6) {};
    \node[above=2pt] at (p31.north) {$p_{3,1}$};

    \node[draw, circle, inner sep=2pt] (p32) at (4,-1.6) {};
    \node[above=2pt] at (p32.north) {$p_{3,2}$};

    \node[draw, circle, inner sep=2pt] (p33) at (6,-1.6) {};
    \node[above=2pt] at (p33.north) {$p_{3,3}$};
   
    % Horizontal Edges
    \draw (p10) -- (p11) -- (p12) -- (p13);
    \draw (p20) -- (p21) -- (p22) -- (p23);
    \draw (p30) -- (p31) -- (p32) -- (p33);

    % Vertical and Diagonal Edges
    
    \draw (p11) -- (p20);
    \draw (p11) -- (p30);
    \draw (p11) .. controls (1.5, 1) and (1.5, -1) ..  (p31);
    \draw (p12) -- (p22);
    \draw (p12) -- (p23);
    \draw (p12) .. controls (4.5, 1) and (4.5, -1) ..  (p32);
    \draw (p12) -- (p33);

\end{tikzpicture}
    \caption{$DDPG_{3,3}$.}
        \label{fig:DDPG33}
\end{figure}
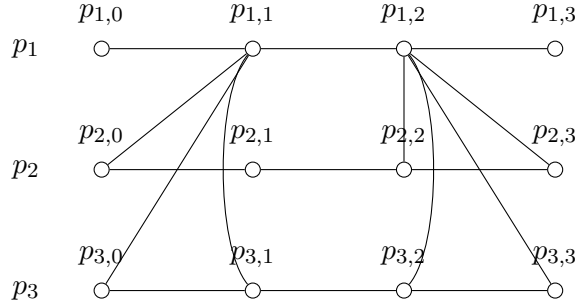

\begin{theorem}
$DDPG_{k,3}$ is an extremal graph having a set of $k$ totally disjoint diametral paths of length $3$.
\end{theorem}

 \begin{proof}
Every disjoint diametral path $p_i$ must have $3$ edges connecting the $4$ vertices in the path. Removal of one of these edges reduces the number of totally disjoint diametral paths by $1$.  The edges added, connecting other paths to the path $p_1$ maintain the distance between end vertices of other paths.  The connections also make the maximum distance between any pair of vertices to be less than or equal to $3$. Vertex $p_{2,1}$ is at distance $2$ from $p_{1,1}$ and $p_{1,2}$.  Therefore, $p_{2,1}$ remains at distance less than or equal to $3$ from every other vertex through those two vertices.  Removal of an edge connecting any other vertex $v$ to a vertex in path $p_1$ would result in a diametral path of length $4$ between $v$ and $p_{2,1}$.
 \end{proof}

\begin{corollary}
    $MinE(k,3) =  7k-5$ for $k\geq 2$ and $MinE(1,3) =  3$.
\end{corollary}
\begin{proof}
The disjoint diametral paths together have $3k$ edges.  When $k>1$, $k-2$ of the edges each have $4$ other edges connecting the path to $p_1$, giving $4k-8$ edges.  Finally, edge $p_2$ has $3$ other edges connecting it to path $p_1$.
\end{proof}

In all of the extremal graphs defined in this section, the vertex disjoint diametral paths partition the vertices of the graph into sets of size $d+1$ for diameter $d$.

%% file: 05-Conclusion.tex
In this paper, we introduce the study of totally disjoint diametral paths in connected graphs.   First, we show that determining the maximum size of a set of totally disjoint diametral paths in an arbitrary graph is $NP$-complete. We explore the maximum sizes of sets of totally disjoint diametral paths in several classes of graphs, presenting efficient algorithms for 2-paths and threshold graphs, and structural bounds for proper interval graphs. Finally, we present classes of extremal graphs with $k$ totally disjoint diametral paths of length $d$ having the fewest possible number of edges. 

The research reported here leaves several open problems. Determining $\#TDDP(Q_n)$ for the hypercube of dimension $n\geq 7$ is an open question. Finding other classes of graphs $G$ for which $\#TDDP(G)$ can be determined efficiently is also open for future work. In particular, determining whether $\#TDDP$ can be computed in polynomial time for proper interval graphs, for (maximal) outerplanar graphs, or for graphs of bounded treewidth would be of interest. The work done here on 2-paths suggests an approach to solve the problem for outerplanar graphs.